\pdfoutput=1
\RequirePackage{ifpdf}
\ifpdf 
\documentclass[pdftex]{sigma}
\else
\documentclass{sigma}
\fi

\usepackage{overpic}

\numberwithin{equation}{section}

\newtheorem{Theorem}{Theorem}[section]
\newtheorem*{Theorem*}{Theorem}
\newtheorem{Corollary}[Theorem]{Corollary}
\newtheorem{Lemma}[Theorem]{Lemma}
\newtheorem{Proposition}[Theorem]{Proposition}
 { \theoremstyle{definition}
\newtheorem{Definition}[Theorem]{Definition}

\newtheorem{Example}[Theorem]{Example}
\newtheorem{Remark}[Theorem]{Remark} }

\newcommand{\R}{\mathbb{R}}
\newcommand{\id}{\mathrm{id}}

\begin{document}

\allowdisplaybreaks

\renewcommand{\thefootnote}{}

\newcommand{\arXivNumber}{2306.04015}

\renewcommand{\PaperNumber}{035}

\FirstPageHeading

\ShortArticleName{Scalar Curvature Rigidity of Warped Product Metrics}

\ArticleName{Scalar Curvature Rigidity of Warped Product Metrics\footnote{This paper is a~contribution to the Special Issue on Differential Geometry Inspired by Mathematical Physics in honor of Jean-Pierre Bourguignon for his 75th birthday. The~full collection is available at \href{https://www.emis.de/journals/SIGMA/Bourguignon.html}{https://www.emis.de/journals/SIGMA/Bourguignon.html}}}

\Author{Christian B\"AR~$^{\rm a}$, Simon BRENDLE~$^{\rm b}$, Bernhard HANKE~$^{\rm c}$ and Yipeng WANG~$^{\rm b}$}

\AuthorNameForHeading{C.~B\"ar, S.~Brendle, B.~Hanke and Y.~Wang}

\Address{$^{\rm a)}$~Institut f\"ur Mathematik, Universit\"at Potsdam, 14476 Potsdam, Germany}
\EmailD{\href{mailto:cbaer@uni-potsdam.de}{cbaer@uni-potsdam.de}}
\URLaddressD{\url{https://www.math.uni-potsdam.de/baer}}

\Address{$^{\rm b)}$~Department of Mathematics, Columbia University, New York NY 10027, USA}
\EmailD{\href{mailto:simon.brendle@columbia.edu}{simon.brendle@columbia.edu}, \href{mailto:yw3631@columbia.edu}{yw3631@columbia.edu}}

\Address{$^{\rm c)}$~Institut f\"ur Mathematik, Universit\"at Augsburg, 86135 Augsburg, Germany}
\EmailD{\href{mailto:hanke@math.uni-augsburg.de}{hanke@math.uni-augsburg.de}}
\URLaddressD{\url{https://www.math.uni-augsburg.de/diff/hanke}}

\ArticleDates{Received June 09, 2023, in final form April 08, 2024; Published online April 18, 2024}

\Abstract{We show scalar-mean curvature rigidity of warped products of round spheres of dimension at least~2 over compact intervals equipped with strictly log-concave warping functions. This generalizes earlier results of Cecchini--Zeidler to all dimensions. Moreover, we show scalar curvature rigidity of round spheres of dimension at least~3 with two antipodal points removed. This resolves a problem in Gromov's ``Four Lectures'' in all dimensions. Our arguments are based on spin geometry.}


\Classification{53C20; 53C21; 53C27}

\begin{flushright}
\begin{minipage}{65mm}
\it
Dedicated to Jean-Pierre Bourguignon\\
 on the occasion of his 75th birthday
 \end{minipage}
\end{flushright}

\renewcommand{\thefootnote}{\arabic{footnote}}
\setcounter{footnote}{0}

\section{Introduction}

In this paper, we study rigidity results for metrics with lower scalar curvature bounds.
One of the first results of this kind is the famous rigidity theorem of Llarull \cite{Llarull}.
Let $g_{S^n}$ denote the standard round metric on $S^n$ with scalar curvature $n(n-1)$. Llarull showed that, if $g$ is a metric on $S^n$ with $g \geq g_{S^n}$ and $R_g \geq n(n-1)$, then $g=g_{S^n}$.
The proof of Llarull's theorem uses Dirac operator techniques in an ingenious way, and is inspired by the fundamental work of Gromov and Lawson \cite{Gromov-Lawson1, Gromov-Lawson2}.

In an important paper, Cecchini and Zeidler extended this line of thought and proved scalar and mean curvature rigidity results for odd-dimensional manifolds with boundary where the comparison metric is not the round metric, but a warped product metric (see \cite[Section~10]{Cecchini-Zeidler}).
In the first part of the present paper we will remove the dimension parity assumption in some of their results.

Given a Riemannian manifold $(M,g)$ with boundary, we denote by $R_g$ the scalar curvature of $g$. Moreover, we denote by $\nu_g$ the outward unit normal with respect to $g$.
We denote by $H_g$ the mean curvature of $\partial M$ with respect to $g$, defined as the sum of the principal curvatures.
The sign convention for $H_g$ is such that the mean curvature vector is given by $-H_g\nu_g$.

Let $n>2$, let $\theta_- < \theta_+$ and let $\rho\colon [\theta_- , \theta_+] \to \R$ be a positive smooth function.
We consider the warped product metric
\begin{equation} \label{g_0}
 g_0 = {\rm d}\theta \otimes {\rm d}\theta + \rho^2(\theta) g_{S^{n-1}}
\end{equation}
on $S^{n-1} \times [\theta_- , \theta_+]$. The scalar curvature of $g_0$ is given by
\begin{equation} \label{scal_g_0}
 R_{g_0} = (n-1) \left( -2 \frac{\rho''(\theta)}{\rho(\theta)} + (n-2) \frac{1- \rho'(\theta)^2}{\rho(\theta)^2} \right),
\end{equation}
while the boundary mean curvature of $g_0$ is given by
\begin{equation} \label{mean_g_0}
 H_{g_0} = \pm (n-1) \frac{\rho'(\theta_{\pm})}{\rho(\theta_{\pm})}\qquad \text{along}\quad S^{n-1} \times \{ \theta_{\pm} \}
\end{equation}
(cf.\ \cite[Example~4.1]{Baer-Gauduchon-Moroianu}).

Our first result says that warped product metrics satisfy a scalar-mean curvature rigidity property, provided that the warping function is strictly logarithmically concave.

\begin{theorem} \label{spherical_band}
Let $n>2$, let $\rho\colon [\theta_- , \theta_+] \to \R$ be a positive smooth function such that ${(\log \rho)''<0}$.
Let $g_0$ denote the warped product metric in \eqref{g_0}.
Let $M$ be a compact, connected spin manifold of dimension $n$ with boundary $\partial M$.
Let $g$ be a Riemannian metric on $M$.
Suppose that $\Phi\colon (M,g) \to \big(S^{n-1} \times [\theta_-,\theta_+],g_0\big)$ is a smooth map with the following properties:
\begin{itemize}\itemsep=0pt
\item $\Phi(\partial M) \subset S^{n-1} \times \{\theta_+,\theta_-\}$,
\item $\Phi$ has non-zero degree,
\item $\Phi$ is $1$-Lipschitz,
\item $R_g \geq R_{g_0} \circ \Phi$ at each point in $M$, compare \eqref{scal_g_0},
\item $H_g \geq H_{g_0} \circ \Phi$ at each point in $\partial M$, compare \eqref{mean_g_0}.
\end{itemize}
Then $\Phi$ is a Riemannian isometry.
\end{theorem}

We note that $R_{g_0}$ in this theorem is not required to be non-negative.
For $n$ odd, Theorem~\ref{spherical_band} is implied by results of Cecchini--Zeidler, see \cite[Theorem~10.2]{Cecchini-Zeidler}.

Applying this discussion to annuli in simply-connected space forms as in \cite[Section~10]{Cecchini-Zeidler}, this removes the parity restriction in \cite[Corollaries~10.4 and~10.5]{Cecchini-Zeidler}.

\begin{Example}
If $0 < \theta_- < \theta_+ < \pi$ and $\rho(\theta) = \sin \theta$, then the warped product metric $g_0$ in \eqref{g_0} has constant scalar curvature $R_{g_0} = n(n-1)$. If $\theta_- < \theta_+$ and $\rho(\theta) = \sinh \theta$, then the warped product metric $g_0$ in \eqref{g_0} has constant scalar curvature $R_{g_0} = -n(n-1)$.
\end{Example}

\begin{Example}
The spatial Schwarzschild--de~Sitter metrics on $S^{n-1} \times [\theta_-, \theta_+]$ are rotationally symmetric and have scalar curvature equal to a positive constant.
Similarly, the spatial Schwarzschild--anti--de~Sitter metrics on $S^{n-1} \times [\theta_-, \theta_+]$ are rotationally symmetric and have scalar curvature equal to a negative constant.
These metrics can be expressed as warped products of the form $g_0 = {\rm d}\theta \otimes {\rm d}\theta + \rho^2(\theta) g_{S^{n-1}}$, see, e.g., \cite[p.~64]{Lee}.
If we restrict to an interval where $\log \rho$ is strictly concave, then we obtain the rigidity property in Theorem~\ref{spherical_band}.
\end{Example}

The second theme of our paper is a rigidity result for metrics on the sphere $S^n$ with two antipodal points removed.
This can be viewed as a limiting case of the band rigidity results treated in the first part of our paper.
This is related to a conjecture of Gromov \cite{Gromov}.
He conjectured that Llarull's theorem holds for metrics that are defined on the sphere $S^n$ with finitely many points removed.
In the special case of two antipodal punctures, Gromov sketched an argument based on $\mu$-bubbles (see \cite[Sections 5.5 and 5.7]{Gromov}).
In the three-dimensional case, a detailed proof based on $\mu$-bubbles was given by Hu, Liu, and Shi \cite{Hu-Liu-Shi}.
An alternative proof in the three-dimensional case was given by Hirsch, Kazaras, Khuri, and Zhang \cite{Hirsch-Kazaras-Khuri-Zhang}.
Using Dirac operator techniques, we generalize these results to all dimensions:

\begin{theorem}
\label{punctures}
Let $n>2$.
We consider the warped product metric $g_0 = {\rm d}\theta \otimes {\rm d}\theta + \sin^2 \theta g_{S^{n-1}}$ on~$S^{n-1} \times (0,\pi)$.
Let $\Omega$ be a non-compact, connected spin manifold of dimension $n$ without boundary.
Let $g$ be a $($possibly incomplete$)$ Riemannian metric on $\Omega$ with scalar curvature~${R_g \geq n(n-1)}$.
Suppose that $\Phi\colon (\Omega,g) \to \big(S^{n-1} \times (0,\pi),g_0\big)$ is a smooth map with the following properties:
\begin{itemize}\itemsep=0pt
\item $\Phi$ is proper,
\item $\Phi$ has non-zero degree,
\item $\Phi$ is $1$-Lipschitz.
\end{itemize}
Then $\Phi$ is a Riemannian isometry.
\end{theorem}

\begin{Remark}
Theorems~\ref{spherical_band} and \ref{punctures} do not hold for $n=2$.
To see this, we choose $\lambda > 1$ and consider the metrics $g = {\rm d}\theta \otimes {\rm d}\theta + \lambda \sin^2 (\theta) g_{S^1}$ and $g_0 = {\rm d}\theta \otimes {\rm d}\theta + \sin^2 (\theta) g_{S^1}$ on $S^1 \times (0,\pi)$. Then $R_g=R_{g_0}=2$, and the identity map from $\big(S^1 \times (0,\pi),g\big)$ to $\big(S^1 \times (0,\pi),g_0\big)$ is $1$-Lipschitz, but not an isometry.
\end{Remark}

Our argument relies on the spin geometric approach to scalar curvature rigidity as introduced in \cite{Llarull} and further developed in \cite{Cecchini-Zeidler}.
A new feature of the present work is the construction of non-zero harmonic spinor fields for which the right-hand side of the integral Schr\"odinger--Lichnerowicz--Weitzenb\"ock formula has a favorable sign, but which cannot be generated directly by index-theoretic arguments.
This construction uses limits of sequences of non-zero harmonic spinor fields whose existence follows from index theory, cf.\ Corollaries~\ref{limiting_spinor_even_dim} and~\ref{special_spinor_on_Omega}.

In contrast to \cite{Cecchini-Zeidler}, our index calculations take place exclusively on compact manifolds.
The corresponding ``holographic'' index theorem for compact manifolds with boundary is formulated and proved in Appendix~\ref{holographic}, which may be of independent interest.

After this paper was written, we learned of a preprint by Wang and Xie \cite{Wang-Xie} announcing similar results.

\section{Proof of Theorem~\ref{spherical_band}}\label{proof_of_band_comparison}

\subsection[Proof of Theorem~\ref{spherical_band} for n even]{Proof of Theorem~\ref{spherical_band} for $\boldsymbol{n}$ even}

We first prove Theorem~\ref{spherical_band} for even $n$, which is not treated in \cite{Cecchini-Zeidler}.
The necessary changes in the odd-dimensional case will be explained in the next section.

Fix an even integer $n>2$ and a warping function $\rho\colon [\theta_- , \theta_+] \to \R$ such that $(\log \rho)''<0$.
Let~$g_0$ denote the warped product metric in \eqref{g_0}.
Let $M$ be a compact, connected spin manifold of dimension $n$ with boundary $\partial M$.
Let $g$ be a Riemannian metric on $M$.
Suppose that~${\Phi\colon (M,g) \to \big(S^{n-1} \times [\theta_-,\theta_+],g_0\big)}$ is a smooth map satisfying the assumptions of Theorem~\ref{spherical_band}.

Let $\varphi\colon M \to S^{n-1}$ denote the projection of $\Phi$ to the first factor, and let $\Theta\colon M \to [\theta_-,\theta_+]$ denote the projection of $\Phi$ to the second factor. Since $\Phi = (\varphi,\Theta)$ is $1$-Lipschitz, we obtain
\begin{equation} \label{ineq_for_g}
 g \geq {\rm d}\Theta \otimes {\rm d}\Theta + \rho^2(\Theta) \varphi^* g_{S^{n-1}}.
\end{equation}

\begin{Lemma}
\label{gradient_of_Theta}
We have $|\nabla \Theta| \leq 1$, and the inequality is strict unless ${\rm d}\varphi(\nabla \Theta) = 0$.
\end{Lemma}

\begin{proof}
Evaluating the inequality \eqref{ineq_for_g} at the vector $\nabla \Theta$ gives
\[|\nabla \Theta|^2 \geq |\nabla \Theta|^4 + \rho^2(\Theta) |{\rm d}\varphi(\nabla \Theta)|_{g_{S^{n-1}}}^2.\]
From this, the assertion follows easily.
\end{proof}

We write $\partial M = \partial_+ M \cup \partial_- M$, where
\[
\partial_+ M := \partial M \cap \Theta^{-1}(\{\theta_+\}), \qquad \partial_- M := \partial M \cap \Theta^{-1}(\{\theta_-\}).
\]

\begin{Lemma} \label{deg_of_restriction}
Let $\varphi|_{\partial_- M}\colon \partial_- M \to S^{n-1}$ denote the restriction of $\varphi$ to $\partial_- M$.
Then $\deg(\Phi) = \pm \deg(\varphi|_{\partial_- M})$.
\end{Lemma}

\begin{proof}
We can find a smooth function $\hat{\Theta}\colon M \to [\theta_-,\theta_+]$ such that $\hat{\Theta}^{-1}(\{\theta_+\}) = \partial_+ M$, $\hat{\Theta}^{-1}(\{\theta_-\}) = \partial_- M$, and ${\rm d}\hat{\Theta} \neq 0$ at each point on $\partial_+ M \cup \partial_- M$.
Let us define a map $\hat{\Phi}\colon M \to S^{n-1} \times [\theta_-,\theta_+]$ by $\hat{\Phi} = \big(\varphi,\hat{\Theta}\big)$. Clearly, $\deg\big(\hat{\Phi}\big) = \pm \deg(\varphi|_{\partial_- M})$.
Since $\hat{\Phi}$ is homotopic to $\Phi$ relative to $\partial M$, the assertion follows.
\end{proof}

For even $n$, the boundary $\partial M$ is odd-dimensional which is inconvenient for the index calculations.
As in \cite{Llarull}, we remedy the situation by considering products with circles of large radius and sending the radius to infinity.
Let $r$ be a positive real number. We consider the product~${\tilde{M} = M \times S^1}$ equipped with the product metric $\tilde{g} = g + r^2 g_{S^1}$. We write $\partial \tilde{M} = \partial_+ \tilde{M} \cup \partial_- \tilde{M}$, where
\[\partial_+ \tilde{M} := \partial_+ M \times S^1, \qquad \partial_- \tilde{M} := \partial_- M \times S^1.\]

\begin{Lemma} \label{map_to_Sn}
There exists a smooth map
\begin{equation*}
 h\colon\ S^{n-1} \times S^1 \to S^n
\end{equation*}
of degree $\pm 1$ with the property that $h^* g_{S^n} \leq g_{S^{n-1}} + 4 g_{S^1}$.
\end{Lemma}

\begin{proof}
Fix a smooth $2$-Lipschitz function $\beta\colon [-\pi,\pi] \to [-\pi,\pi]$ such that $\beta(t) = -\pi$ for $t \in \bigl[-\pi,-\frac{7\pi}{8}\bigr]$, $\beta(t) = 0$ for $t \in \bigl[-\frac{\pi}{8},\frac{\pi}{8}\bigr]$, and $\beta(t) = \pi$ for $t \in \big[\frac{7\pi}{8},\pi\big]$. Moreover, let us fix a point $a \in S^{n-1}$. We consider the map
\[
 S^{n-1} \times [-\pi,\pi] \to S^n,\qquad (x,t) \mapsto \begin{cases} (\sin \beta(t) x,\cos \beta(t)) & \text{\rm for $t \in [-\pi,0]$}, \\
 (\sin \beta(t) a,\cos \beta(t)) & \text{\rm for $t \in [0,\pi]$.} \end{cases}
 \]
This gives a map $h\colon S^{n-1} \times S^1 \to S^n$ of degree $\pm 1$. Moreover, $h^* g_{S^n} = \sin^2 \beta(t) g_{S^{n-1}} + \beta'(t)^2 g_{S^1}$ for $t \in [-\pi,0]$ and $h^* g_{S^n} = \beta'(t)^2 g_{S^1}$ for $t \in [0,\pi]$.
\end{proof}

In the following, we assume that $h\colon S^{n-1} \times S^1 \to S^n$ is chosen as in Lemma~\ref{map_to_Sn}. We define a smooth map $\tilde{f}\colon \tilde{M} = M \times S^1\to S^n$, $
\tilde{f}(x,t) = h(\varphi(x),t)$ for $x \in M$ and $t \in S^1$.

Choose a spin structure on $M$ and let $S$ denote the spinor bundle over $M$. Furthermore, let $\tilde{S}$ denote the spinor bundle over $\tilde{M} = M \times S^1$, where $S^1$ is equipped with the trivial spin structure $S^1 \times {\rm Spin}(1) \to S^1 \times {\rm SO}(1)$.
Note that with this choice, $\tilde{S}$ is the pull-back of $S$, as a~Clifford-module bundle, under the projection from $\tilde{M} = M \times S^1$ to $M$.

Let $E_0$ denote the spinor bundle of the round sphere $S^n$.
The bundle $E_0$ is equipped with a preferred bundle metric and connection.
Since $n$ is even, we may decompose $E_0$ in the usual way as $E_0 = E_0^+ \oplus E_0^-$, where $E_0^+$ and $E_0^-$ are the $\pm 1$-eigenbundles of the complex volume form.

Next we need an index computation.

\begin{Proposition}[cf.\ Cecchini--Zeidler \cite{Cecchini-Zeidler}]
\label{index_even_dim}
Consider the indices of the following operators:
\begin{itemize}\itemsep=0pt
\item Let $\mathrm{ind}_1$ denote the index of the Dirac operator on $\tilde{S} \otimes \tilde{f}^* E_0^+$ with boundary conditions $u = -{\rm i} \nu \cdot u$ on $\partial_+ \tilde{M}$ and $u = {\rm i} \nu \cdot u$ on $\partial_- \tilde{M}$.
\item Let $\mathrm{ind}_2$ denote the index of the Dirac operator on $\tilde{S} \otimes \tilde{f}^* E_0^+$ with boundary conditions $u = {\rm i} \nu \cdot u$ on $\partial_+ \tilde{M}$ and $u = -{\rm i} \nu \cdot u$ on $\partial_- \tilde{M}$.
\item Let $\mathrm{ind}_3$ denote the index of the Dirac operator on $\tilde{S} \otimes \tilde{f}^* E_0^-$ with boundary conditions $u = -{\rm i} \nu \cdot u$ on $\partial_+ \tilde{M}$ and $u = {\rm i} \nu \cdot u$ on $\partial_- \tilde{M}$.
\item Let $\mathrm{ind}_4$ denote the index of the Dirac operator on $\tilde{S} \otimes \tilde{f}^* E_0^-$ with boundary conditions $u = {\rm i} \nu \cdot u$ on $\partial_+ \tilde{M}$ and $u = -{\rm i} \nu \cdot u$ on $\partial_- \tilde{M}$.
\end{itemize}
Then $\mathrm{ind}_1+\mathrm{ind}_2 = 0$, $\mathrm{ind}_3+\mathrm{ind}_4 = 0$, and $\max \{\mathrm{ind}_1,\mathrm{ind}_2, \mathrm{ind}_3, \mathrm{ind}_4\} > 0$.
\end{Proposition}

\begin{proof}
Since the boundary conditions are adjoint to each other, we obtain $\mathrm{ind}_1+\mathrm{ind}_2 = 0$ and~${\mathrm{ind}_3+\mathrm{ind}_4 = 0}$.

It remains to show that $\max \{\mathrm{ind}_1,\mathrm{ind}_2, \mathrm{ind}_3, \mathrm{ind}_4\} > 0$.
Suppose that this is false.
Then $\mathrm{ind}_1=\mathrm{ind}_2=\mathrm{ind}_3=\mathrm{ind}_4=0$.
We will apply the holographic index theorem in Appendix~\ref{holographic} and the Atiyah--Singer index theorem to show that the assumption $\mathrm{ind}_1=\mathrm{ind}_3=0$ already leads to a contradiction.

The restriction \smash{$\tilde{S}|_{\partial_- \tilde{M}}$} can be identified with the spinor bundle on $\partial_- \tilde{M}$.
We may write \smash{$\tilde{S}|_{\partial_- \tilde{M}} = S^+ \oplus S^-$}, where $S^+$ and $S^-$ denote the eigenbundles of the volume form on $\partial_- \tilde{M}$.
Equivalently, $S^+$ and $S^-$ can be characterized as the eigenbundles of $i\nu$.
This gives the splitting
\[
 \big(\tilde{S} \otimes \tilde{f}^* E_0^+\big)|_{\partial_- \tilde{M}} = \big(S^+ \otimes \big(\tilde{f}|_{\partial_- \tilde{M}}\big)^* E_0^+\big) \oplus \big(S^- \otimes \big(\tilde{f}|_{\partial_- \tilde{M}}\big)^* E_0^+\big).
\]
Similarly,
\[
 \big(\tilde{S} \otimes \tilde{f}^* E_0^-\big)|_{\partial_- \tilde{M}} = \big(S^+ \otimes \big(\tilde{f}|_{\partial_- \tilde{M}}\big)^* E_0^-\big) \oplus \big(S^- \otimes \big(\tilde{f}|_{\partial_- \tilde{M}}\big)^* E_0^-\big).
 \]
Since $\mathrm{ind}_1 =0$, Corollary~\ref{cor:Freed} tells us that the boundary Dirac operator which maps sections of \smash{$S^+ \otimes \big(\tilde{f}|_{\partial_- \tilde{M}}\big)^* E_0^+$} to sections of \smash{$S^- \otimes \big(\tilde{f}|_{\partial_- \tilde{M}}\big)^* E_0^+$} has index $0$.
Similarly, since $\mathrm{ind}_3=0$, the boundary Dirac operator which maps sections of \smash{$S^+ \otimes \big(\tilde{f}|_{\partial_- \tilde{M}}\big)^* E_0^-$} to sections of $\smash{S^- \otimes} \allowbreak\smash{\big(\tilde{f}|_{\partial_- \tilde{M}}\big)^* E_0^-}$ has index $0$.

To obtain a contradiction, we compute the index of the boundary Dirac operators using the Atiyah--Singer index theorem.
Denote the total $\hat{\mathsf{A}}$-class of $\partial_- \tilde{M}$ by $\hat{\mathsf{A}}\big(\partial_- \tilde{M}\big)$. The Chern character of the bundle \smash{$\big(\tilde{f}|_{\partial_- \tilde{M}}\big)^* E_0^+$} is given by the pull-back of $\mathsf{ch}\big(E_0^+\big)$ under \smash{$\tilde{f}|_{\partial_- \tilde{M}}$}. In particular, the Chern character of the bundle \smash{$\big(\tilde{f}|_{\partial_- \tilde{M}}\big)^* E_0^+$} only contains terms in the $0$-th and $n$-th cohomology groups. Since the boundary Dirac operator which maps sections of \smash{$S^+ \otimes \big(\tilde{f}|_{\partial_- \tilde{M}}\big)^* E_0^+$} to sections of \smash{$S^- \otimes \big(\tilde{f}|_{\partial_- \tilde{M}}\big)^* E_0^+$} has index $0$, the Atiyah--Singer index theorem gives
\begin{align}
0
&= \big\langle \hat{\mathsf{A}}\big(\partial_- \tilde{M}\big) \cup \mathsf{ch}\big(\big( \tilde{f}|_{\partial_- \tilde{M}}\big)^* E_0^+\big),\big[\partial_- \tilde{M}\big] \big\rangle\nonumber \\
&= \dim E_0^+ \cdot \big\langle \hat{\mathsf{A}}\big(\partial_- \tilde{M}\big),\big[\partial_- \tilde{M}\big] \big\rangle + \big\langle \mathsf{ch}\big(\big( \tilde{f}|_{\partial_- \tilde{M}}\big)^* E_0^+\big), \big[\partial_- \tilde{M}\big] \big\rangle\nonumber \\
&= \dim E_0^+ \cdot \big\langle \hat{\mathsf{A}}\big(\partial_- \tilde{M}\big),\big[\partial_- \tilde{M}\big] \big\rangle + \deg\big(\tilde{f}|_{\partial_- \tilde{M}}\big) \cdot \big\langle \mathsf{ch}\big(E_0^+\big), [S^n] \big\rangle. \label{index.E_0^+}
\end{align}
Working with $E_0^-$ instead of $E_0^+$, we similarly obtain
\begin{align}
0
&= \big\langle \hat{\mathsf{A}}\big(\partial_- \tilde{M}\big) \cup \mathsf{ch}\big(\big( \tilde{f}|_{\partial_- \tilde{M}}\big)^* E_0^-\big),\big[\partial_- \tilde{M}\big] \big\rangle\nonumber \\
&= \dim E_0^- \cdot \big\langle \hat{\mathsf{A}}\big(\partial_- \tilde{M}\big),\big[\partial_- \tilde{M}\big] \big\rangle + \big\langle \mathsf{ch}\big(\big( \tilde{f}|_{\partial_- \tilde{M}}\big)^* E_0^-\big), \big[\partial_- \tilde{M}\big] \big\rangle \nonumber\\
&= \dim E_0^- \cdot \big\langle \hat{\mathsf{A}}\big(\partial_- \tilde{M}\big),\big[\partial_- \tilde{M}\big] \big\rangle + \deg\big(\tilde{f}|_{\partial_- \tilde{M}}\big) \cdot \big\langle \mathsf{ch}(E_0^-), [S^n] \big\rangle.\label{index.E_0^-}
\end{align}
In the next step, we subtract \eqref{index.E_0^-} from \eqref{index.E_0^+}.
Using the fact that $\dim E_0^+ = \dim E_0^-$, we obtain
\[
 0 = \deg\big(\tilde{f}|_{\partial_- \tilde{M}}\big) \cdot \big\langle \mathsf{ch}\big(E_0^+\big) - \mathsf{ch}(E_0^-), [S^n] \big\rangle.
 \]
It follows from \cite[Proposition~11.24, Chapter~III]{Lawson-Michelsohn} that $\big\langle \mathsf{ch}\big(E_0^+\big) - \mathsf{ch}(E_0^-), [S^n] \big\rangle = \pm \chi(S^n) =\pm 2 \neq 0$ since $n$ is even.
Thus $\deg\big(\tilde{f}|_{\partial_- \tilde{M}}\big)=0$.

By assumption, the map $\Phi\colon M \to S^{n-1} \times [\theta_-,\theta_+]$ has non-zero degree.
Hence, it follows from Lemma~\ref{deg_of_restriction} that the map $\varphi|_{\partial_- M}\colon \partial_- M \to S^{n-1}$ has non-zero degree.
Consequently, the map $\varphi|_{\partial_- M} \times \id\colon \partial_- M \times S^1 \to S^{n-1} \times S^1$ has non-zero degree. By Lemma~\ref{map_to_Sn}, the map~${h\colon S^{n-1} \times S^1 \to S^n}$ has non-zero degree.
Since \smash{$\tilde{f}|_{\partial_- \tilde{M}} = h \circ (\varphi|_{\partial_- M} \times \id)$}, we conclude that the map \smash{$\tilde{f}|_{\partial_- \tilde{M}}\colon \partial_- \tilde{M} \to S^n$} has non-zero degree.
This is a contradiction.
\end{proof}

By Proposition~\ref{index_even_dim}, we know that $\max \{\mathrm{ind}_1,\mathrm{ind}_2, \mathrm{ind}_3, \mathrm{ind}_4\} > 0$. After switching the bundles~$E_0^+$ and $E_0^-$ if necessary, we may assume that $\max \{\mathrm{ind}_1,\mathrm{ind}_2\} > 0$.
In the remainder of this section, we focus on the case $\mathrm{ind}_1 > 0$.
(The case $\mathrm{ind}_2 > 0$ can be treated analogously.)

Let $\tilde{E}$ denote the pull-back of $E_0^+$ under the map $\tilde{f}$. The bundle metric on $E_0^+$ gives us a~bundle metric on $\tilde{E}$. Moreover, the connection on $E_0^+$ induces a connection on $\tilde{E}$. We denote by \smash{$\nabla^{\tilde{S} \otimes \tilde{E}}$} the tensor product connection on $\tilde{S} \otimes \tilde{E}$. We denote by \smash{$\mathcal{D}^{\tilde{S} \otimes \tilde{E}}$} the Dirac operator acting on sections of $\tilde{S} \otimes \tilde{E}$,
\[
 \mathcal{D}^{\tilde{S} \otimes \tilde{E}} u = \sum_{k=1}^{n+1} e_k \cdot \nabla_{e_k}^{\tilde{S} \otimes \tilde{E}} u,
\]
where $\{e_1, \ldots, e_{n+1}\}$ is a local orthonormal frame on $\tilde{M}$.
Finally, we define the boundary Dirac operator by
\[
 \mathcal{D}^{\partial \tilde{M}} u = \sum_{k=1}^n \nu \cdot e_k \cdot \nabla_{e_k}^{\tilde{S} \otimes \tilde{E}} u + \frac{1}{2} H u,
\]
where $\{e_1,\hdots,e_n\}$ is a local orthonormal frame on $\partial \tilde{M}$.
The boundary Dirac operator is self-adjoint and anti-commutes with Clifford multiplication by $\nu$.

Recall the Weitzenb\"ock formula (see \cite[Theorem~8.17, Chapter~II]{Lawson-Michelsohn}),
\[
 \big(\mathcal{D}^{\tilde{S} \otimes \tilde{E}}\big)^2 u
 =
 \big(\nabla^{\tilde{S} \otimes \tilde{E}}\big)^*\nabla^{\tilde{S} \otimes \tilde{E}} u + \frac{1}{4} R u + \mathcal{R}^{\tilde{E}} u,
\]
where $\mathcal{R}^{\tilde{E}}$ is a section of the endomorphism bundle of $\tilde{S} \otimes \tilde{E}$ which depends on the curvature of the bundle $\tilde{E}$.

We define a vector field $T$ on $\tilde{M}$ by $T = \frac{1}{r} \frac{\partial}{\partial t}$, where $t \mapsto (\cos t,\sin t)$ is the canonical local coordinate on $S^1$. Note that $T$ is parallel and has unit length with respect to the metric $\tilde{g}$. In the following, $\Psi$ will denote a smooth function on $M$ which will be specified later. We may extend $\Psi$ to a smooth function on $\tilde{M}$ satisfying $T(\Psi) = 0$. If $u$ is a section of the bundle $\tilde{S} \otimes \tilde{E}$ and $X$ is a vector field on $\tilde{M}$, we define
\begin{equation} \label{def:tildeP}
\tilde{P}_X u = \nabla_{X - \langle X,T \rangle T}^{\tilde{S} \otimes \tilde{E}} u + \frac{\rm i}{2} \Psi \cdot (X - \langle X,T \rangle T) \cdot u.
\end{equation}
Our argument is based on the following integral formula which links several geometric quantities on $\tilde{M}$ and $\partial \tilde{M}$.

\begin{Proposition}
\label{integral_formula_even_dim}
Let $u \in C^\infty\big(\tilde{M},\tilde{S} \otimes \tilde{E}\big)$. Then
\begin{align*}
& - \int_{\tilde{M}} \left|\mathcal{D}^{\tilde{S} \otimes \tilde{E}} u - \frac{{\rm i}n}{2} \Psi u\right|^2 + \int_{\tilde{M}} \big|\tilde{P} u\big|^2 + \int_{\tilde{M}} \big|\nabla_T^{\tilde{S} \otimes \tilde{E}} u\big|^2 + \frac{1}{4} \int_{\tilde{M}} R |u|^2 \\
&+ \int_{\tilde{M}} \big\langle \mathcal{R}^{\tilde{E}} u,u \big\rangle + \frac{n(n-1)}{4} \int_{\tilde{M}} \Psi^2 |u|^2 - \frac{{\rm i}(n-1)}{2} \int_{\tilde{M}} \langle (\nabla \Psi) \cdot u,u \rangle \\
&\qquad= \frac{1}{2} \int_{\partial_+ \tilde{M}} \big\langle \mathcal{D}^{\partial \tilde{M}} u,u + {\rm i} \nu \cdot u \big\rangle + \frac{1}{2} \int_{\partial_+ \tilde{M}} \big\langle u + {\rm i} \nu \cdot u,\mathcal{D}^{\partial \tilde{M}} u \big\rangle \\
&\phantom{\qquad=}{} - \frac{1}{2} \int_{\partial_+ \tilde{M}} (H-(n-1)\Psi) |u|^2 - \frac{n-1}{2} \int_{\partial_+ \tilde{M}} \Psi \langle u + {\rm i} \nu \cdot u,u \rangle \\
&\phantom{\qquad=}{} + \frac{1}{2} \int_{\partial_- \tilde{M}} \big\langle \mathcal{D}^{\partial \tilde{M}} u,u - {\rm i} \nu \cdot u \big\rangle + \frac{1}{2} \int_{\partial_- \tilde{M}} \big\langle u - {\rm i} \nu \cdot u,\mathcal{D}^{\partial \tilde{M}} u \big\rangle \\
&\phantom{\qquad=}{} - \frac{1}{2} \int_{\partial_- \tilde{M}} (H+(n-1)\Psi) |u|^2 + \frac{n-1}{2} \int_{\partial_- \tilde{M}} \Psi \langle u - {\rm i} \nu \cdot u,u \rangle.
\end{align*}
\end{Proposition}

\begin{proof}
Integrating the Weitzenb\"ock formula and using the divergence theorem gives
\begin{align*}
&- \int_{\tilde{M}} \big|\mathcal{D}^{\tilde{S} \otimes \tilde{E}} u\big|^2 + \int_{\tilde{M}} \big|\nabla^{\tilde{S} \otimes \tilde{E}} u\big|^2 + \frac{1}{4} \int_{\tilde{M}} R |u|^2 + \int_{\tilde{M}} \big\langle \mathcal{R}^{\tilde{E}} u,u \big\rangle \\
&\qquad= \int_{\partial \tilde{M}} \big\langle \nu \cdot \mathcal{D}^{\tilde{S} \otimes \tilde{E}} u,u \big\rangle + \int_{\partial \tilde{M}} \big\langle \nabla^{\tilde{S} \otimes \tilde{E}}_\nu u,u \big\rangle.
\end{align*}
Note that $\nu \cdot \mathcal{D}^{\tilde{S} \otimes \tilde{E}} u + \nabla^{\tilde{S} \otimes \tilde{E}}_\nu u = \mathcal{D}^{\partial \tilde{M}} u - \frac{1}{2} H u$.
This gives
\begin{align*}
&- \int_{\tilde{M}} \big|\mathcal{D}^{\tilde{S} \otimes \tilde{E}} u\big|^2 + \int_{\tilde{M}} \big|\nabla^{\tilde{S} \otimes \tilde{E}} u\big|^2 + \frac{1}{4} \int_{\tilde{M}} R |u|^2 + \int_{\tilde{M}} \big\langle \mathcal{R}^{\tilde{E}} u,u \big\rangle \\
&\qquad= \int_{\partial \tilde{M}} \big\langle \mathcal{D}^{\partial \tilde{M}} u,u \big\rangle - \frac{1}{2} \int_{\partial \tilde{M}} H |u|^2 \\
&\qquad= \frac{1}{2} \int_{\partial \tilde{M}} \big\langle \mathcal{D}^{\partial \tilde{M}} u,u \big\rangle + \frac{1}{2} \int_{\partial \tilde{M}} \big\langle u,\mathcal{D}^{\partial \tilde{M}} u \big\rangle - \frac{1}{2} \int_{\partial \tilde{M}} H |u|^2.
\end{align*}
Since $\mathcal{D}^{\partial \tilde{M}}$ is self-adjoint and anti-commutes with $\nu$, we find
\begin{align*}
\int_{\partial_\pm \tilde{M}} \big\langle \mathcal{D}^{\partial \tilde{M}} u,{\rm i} \nu \cdot u \big\rangle
&=
\int_{\partial_\pm \tilde{M}} \big\langle u,\mathcal{D}^{\partial \tilde{M}}({\rm i} \nu \cdot u) \big\rangle =
-\int_{\partial_\pm \tilde{M}} \big\langle u,{\rm i} \nu \cdot\mathcal{D}^{\partial \tilde{M}} u \big\rangle \\
&=
-\int_{\partial_\pm \tilde{M}} \big\langle {\rm i} \nu \cdot u,\mathcal{D}^{\partial \tilde{M}} u \big\rangle .
\end{align*}
Therefore,
\begin{align}
& - \int_{\tilde{M}}\big|\mathcal{D}^{\tilde{S} \otimes \tilde{E}} u\big|^2 + \int_{\tilde{M}} \big|\nabla^{\tilde{S} \otimes \tilde{E}} u\big|^2 + \frac{1}{4} \int_{\tilde{M}} R |u|^2 + \int_{\tilde{M}} \big\langle \mathcal{R}^{\tilde{E}} u,u \big\rangle \nonumber\\
&\qquad= \frac{1}{2} \int_{\partial_+ \tilde{M}} \big\langle \mathcal{D}^{\partial \tilde{M}} u,u + {\rm i} \nu \cdot u \big\rangle + \frac{1}{2} \int_{\partial_+ \tilde{M}} \big\langle u + {\rm i} \nu \cdot u,\mathcal{D}^{\partial \tilde{M}} u \big\rangle\nonumber \\
&\qquad\phantom{=}{} + \frac{1}{2} \int_{\partial_- \tilde{M}} \big\langle \mathcal{D}^{\partial \tilde{M}} u,u - {\rm i} \nu \cdot u \big\rangle + \frac{1}{2} \int_{\partial_- \tilde{M}} \big\langle u - {\rm i} \nu \cdot u,\mathcal{D}^{\partial \tilde{M}} u \big\rangle
 - \frac{1}{2} \int_{\partial \tilde{M}} H |u|^2.\label{eq:formula1}
\end{align}
Using the definition of $\tilde{P} u$ and a local orthonormal frame $e_1,\dots,e_n,T$ on $\tilde{M}$, we compute
\begin{align}
\big|\tilde{P} u\big|^2
={}&
\sum_{k=1}^n \left| \nabla_{e_k}^{\tilde{S} \otimes \tilde{E}} u + \frac{\rm i}{2} \Psi e_k \cdot u \right|^2 \nonumber\\
={}&
\big|\nabla^{\tilde{S} \otimes \tilde{E}} u\big|^2 - \big|\nabla_T^{\tilde{S} \otimes \tilde{E}} u\big|^2 + \frac{n}{4} \Psi^2 |u|^2 \nonumber\\
& + \frac{\rm i}{2} \Psi \big\langle \mathcal{D}^{\tilde{S} \otimes \tilde{E}} u - T \cdot \nabla_T^{\tilde{S} \otimes \tilde{E}} u,u \big\rangle - \frac{\rm i}{2} \Psi \big\langle u,\mathcal{D}^{\tilde{S} \otimes \tilde{E}} u - T \cdot \nabla_T^{\tilde{S} \otimes \tilde{E}} u \big\rangle \nonumber\\
={}&
\big|\nabla^{\tilde{S} \otimes \tilde{E}} u\big|^2 - \big|\nabla_T^{\tilde{S} \otimes \tilde{E}} u\big|^2 + \frac{n}{4} \Psi^2 |u|^2 \nonumber\\
& + \frac{\rm i}{2} \Psi \big\langle \mathcal{D}^{\tilde{S} \otimes \tilde{E}} u,u \big\rangle - \frac{\rm i}{2} \Psi \big\langle u,\mathcal{D}^{\tilde{S} \otimes \tilde{E}} u \big\rangle - \frac{\rm i}{2} \Psi T(\langle T \cdot u,u \rangle).\label{eq:Pu2}
\end{align}
Using the divergence theorem, we find
\begin{align}
\int_{\tilde{M}} \Psi T(\langle T \cdot u,u \rangle)
={}& \int_{\partial \tilde{M}} \Psi \langle T \cdot u,u \rangle \langle T,\nu \rangle \nonumber\\
& - \int_{\tilde{M}} \Psi \langle T \cdot u,u \rangle \mathrm{div} T - \int_{\tilde{M}} T(\Psi) \langle T \cdot u,u \rangle = 0 .\label{eq:lastterm}
\end{align}
We integrate \eqref{eq:Pu2} over $\tilde{M}$ and insert \eqref{eq:lastterm} and obtain
\begin{align}
\int_{\tilde{M}} \big|\tilde{P} u\big|^2
={}& \int_{\tilde{M}} \big|\nabla^{\tilde{S} \otimes \tilde{E}} u\big|^2 - \int_{\tilde{M}} \big|\nabla_T^{\tilde{S} \otimes \tilde{E}} u\big|^2 + \frac{n}{4} \int_{\tilde{M}} \Psi^2 |u|^2\nonumber \\
& + \frac{\rm i}{2} \int_{\tilde{M}} \Psi \big\langle \mathcal{D}^{\tilde{S} \otimes \tilde{E}} u,u \big\rangle - \frac{\rm i}{2} \int_{\tilde{M}} \Psi \big\langle u,\mathcal{D}^{\tilde{S} \otimes \tilde{E}} u \big\rangle.\label{eq:Pu2neu}
\end{align}
Substituting \eqref{eq:Pu2neu} into \eqref{eq:formula1}, we obtain
\begin{align}
&- \int_{\tilde{M}} \big|\mathcal{D}^{\tilde{S} \otimes \tilde{E}} u\big|^2 + \int_{\tilde{M}} \big|\tilde{P} u\big|^2 + \int_{\tilde{M}} \big|\nabla_T^{\tilde{S} \otimes \tilde{E}} u\big|^2 + \frac{1}{4} \int_{\tilde{M}} R |u|^2
+ \int_{\tilde{M}} \big\langle \mathcal{R}^{\tilde{E}} u,u \big\rangle \nonumber\\
&- \frac{n}{4} \int_{\tilde{M}} \Psi^2 |u|^2
 - \frac{\rm i}{2} \int_{\tilde{M}} \Psi \big\langle \mathcal{D}^{\tilde{S} \otimes \tilde{E}} u,u \big\rangle + \frac{\rm i}{2} \int_{\tilde{M}} \Psi \big\langle u,\mathcal{D}^{\tilde{S} \otimes \tilde{E}} u \big\rangle\nonumber \\
&\qquad= \frac{1}{2} \int_{\partial_+ \tilde{M}} \big\langle \mathcal{D}^{\partial \tilde{M}} u,u + {\rm i} \nu \cdot u \big\rangle + \frac{1}{2} \int_{\partial_+ \tilde{M}} \big\langle u + {\rm i} \nu \cdot u,\mathcal{D}^{\partial \tilde{M}} u \big\rangle\nonumber \\
&\qquad\phantom{=}{} + \frac{1}{2} \int_{\partial_- \tilde{M}} \big\langle \mathcal{D}^{\partial \tilde{M}} u,u - {\rm i} \nu \cdot u \big\rangle + \frac{1}{2} \int_{\partial_- \tilde{M}} \big\langle u - {\rm i} \nu \cdot u,\mathcal{D}^{\partial \tilde{M}} u \big\rangle - \frac{1}{2} \int_{\partial \tilde{M}} H |u|^2.\label{eq:formula2}
\end{align}
Using the divergence theorem, we obtain
\begin{align}
&- \frac{{\rm i}(n-1)}{2} \int_{\tilde{M}} \Psi \big\langle \mathcal{D}^{\tilde{S} \otimes \tilde{E}} u,u \big\rangle + \frac{{\rm i}(n-1)}{2} \int_{\tilde{M}} \Psi \big\langle u,\mathcal{D}^{\tilde{S} \otimes \tilde{E}} u \big\rangle
- \frac{{\rm i}(n-1)}{2} \int_{\tilde{M}} \langle (\nabla \Psi) \cdot u,u \rangle\nonumber \\
&\qquad= - \frac{{\rm i}(n-1)}{2} \int_{\tilde{M}} \big\langle \mathcal{D}^{\tilde{S} \otimes \tilde{E}} (\Psi u),u \big\rangle + \frac{{\rm i}(n-1)}{2} \int_{\tilde{M}} \big\langle \Psi u,\mathcal{D}^{\tilde{S} \otimes \tilde{E}} u \big\rangle\nonumber \\
&\qquad= -\frac{{\rm i}(n-1)}{2} \int_{\partial \tilde{M}} \langle \nu \cdot (\Psi u),u \rangle.\label{eq:formula3}
\end{align}
Adding \eqref{eq:formula2} and \eqref{eq:formula3} gives
\begin{align*}
&- \int_{\tilde{M}} \big|\mathcal{D}^{\tilde{S} \otimes \tilde{E}} u\big|^2 + \int_{\tilde{M}} \big|\tilde{P} u\big|^2 + \int_{\tilde{M}} \big|\nabla_T^{\tilde{S} \otimes \tilde{E}} u\big|^2 + \frac{1}{4} \int_{\tilde{M}} R |u|^2
+ \int_{\tilde{M}} \big\langle \mathcal{R}^{\tilde{E}} u,u \big\rangle \\
&- \frac{n}{4} \int_{\tilde{M}} \Psi^2 |u|^2 - \frac{{\rm i}(n-1)}{2} \int_{\tilde{M}} \langle (\nabla \Psi) \cdot u,u \rangle
- \frac{{\rm i}n}{2} \int_{\tilde{M}} \Psi \big\langle \mathcal{D}^{\tilde{S} \otimes \tilde{E}} u,u \big\rangle + \frac{{\rm i}n}{2} \int_{\tilde{M}} \Psi \big\langle u,\mathcal{D}^{\tilde{S} \otimes \tilde{E}} u \big\rangle\\
&\qquad= \frac{1}{2} \int_{\partial_+ \tilde{M}} \big\langle \mathcal{D}^{\partial \tilde{M}} u,u + {\rm i} \nu \cdot u \big\rangle + \frac{1}{2} \int_{\partial_+ \tilde{M}} \big\langle u + {\rm i} \nu \cdot u,\mathcal{D}^{\partial \tilde{M}} u \big\rangle \\
&\phantom{\qquad=}{} - \frac{1}{2} \int_{\partial_+ \tilde{M}} (H-(n-1)\Psi) |u|^2 - \frac{n-1}{2} \int_{\partial_+ \tilde{M}} \Psi \langle u + {\rm i} \nu \cdot u,u \rangle \\
&\phantom{\qquad=}{} + \frac{1}{2} \int_{\partial_- \tilde{M}} \big\langle \mathcal{D}^{\partial \tilde{M}} u,u - {\rm i} \nu \cdot u \big\rangle + \frac{1}{2} \int_{\partial_- \tilde{M}} \big\langle u - {\rm i} \nu \cdot u,\mathcal{D}^{\partial \tilde{M}} u \big\rangle \\
&\phantom{\qquad=}{} - \frac{1}{2} \int_{\partial_- \tilde{M}} (H+(n-1)\Psi) |u|^2 + \frac{n-1}{2} \int_{\partial_- \tilde{M}} \Psi \langle u - {\rm i} \nu \cdot u,u \rangle.
\end{align*}
This completes the proof of Proposition~\ref{integral_formula_even_dim}.
\end{proof}

At this point, we specify our choice of the function $\Psi$. We define a function $\psi\colon [\theta_-, \theta_+] \to \R$ by $\psi(\theta) := \frac{\rho'(\theta)}{\rho(\theta)}$.
Since $(\log \rho)''<0$, we know that $-\psi' > 0$ on $[\theta_-,\theta_+]$.
Using \eqref{scal_g_0}, the inequality $R \geq R_{g_0} \circ \Phi$ gives
\begin{equation} \label{ineq_for_scal_even_dim}
R \geq (n-1) \left( -2 \psi'(\Theta) - n \psi^2(\Theta) + \frac{n-2}{\rho^2(\Theta)} \right).
\end{equation}
We define $\Psi\colon M \to \R$ by $\Psi = \psi \circ \Theta$.

\begin{Proposition}
\label{existence_of_spinor_even_dim}
Assume that \smash{$r > 2 \sup_{[\theta_-,\theta_+]} \rho$}. Then we can find an element $t_0 \in S^1$ and a~section $u \in C^\infty\big(\tilde{M},\tilde{S} \otimes \tilde{E}\big)$ such that
\[\int_{M \times \{t_0\}} \frac{1}{\rho(\Theta)} |u|^2 = 1\]
and
\[\int_{M \times \{t_0\}} \big|\tilde{P} u\big|^2 \leq \frac{n-1}{r}.\]
\end{Proposition}

\begin{proof}
Recall that we are assuming $\mathrm{ind}_1 > 0$. In view of the deformation invariance of the index, we can find a section $u \in C^\infty\big(\tilde{M},\tilde{S} \otimes \tilde{E}\big)$ such that
\begin{itemize}\itemsep=0pt
 \item $u$ does not vanish identically,
 \item $\mathcal{D}^{\tilde{S} \otimes \tilde{E}} u - \frac{{\rm i}n}{2} \Psi u = 0$ on $\tilde{M}$,
 \item $u = -{\rm i} \nu \cdot u$ on $\partial_+ \tilde{M}$ and $u = {\rm i} \nu \cdot u$ on $\partial_- \tilde{M}$.
\end{itemize}
Using Proposition~\ref{integral_formula_even_dim}, we obtain
\begin{align}
&\int_{\tilde{M}} \big|\tilde{P} u\big|^2 + \int_{\tilde{M}} \big|\nabla_T^{\tilde{S} \otimes \tilde{E}} u\big|^2 + \frac{1}{4} \int_{\tilde{M}} R |u|^2\nonumber \\
&\qquad{}+ \int_{\tilde{M}} \big\langle \mathcal{R}^{\tilde{E}} u,u \big\rangle + \frac{n(n-1)}{4} \int_{\tilde{M}} \Psi^2 |u|^2 - \frac{{\rm i}(n-1)}{2} \int_{\tilde{M}} \langle (\nabla \Psi) \cdot u,u \rangle \nonumber\\
&\phantom{\qquad+}{}= -\frac{1}{2} \int_{\partial_+ \tilde{M}} (H-(n-1)\Psi) |u|^2 - \frac{1}{2} \int_{\partial_- \tilde{M}} (H+(n-1)\Psi) |u|^2.\label{important_estimate_even_dim}
\end{align}
By assumption and using \eqref{mean_g_0}, \[H-(n-1)\Psi = H-(n-1) \frac{\rho'(\theta_+)}{\rho(\theta_+)} \geq 0\] on $\partial_+ \tilde{M}$ and \[H +(n-1)\Psi \geq H+(n-1) \frac{\rho'(\theta_-)}{\rho(\theta_-)} \geq 0\] on $\partial_- \tilde{M}$. Consequently, the right-hand side in \eqref{important_estimate_even_dim} is non-positive.

We next analyze the term $\mathcal{R}^{\tilde{E}}$. To that end, we fix a point $(x,t) \in \tilde{M}$. Let $ \mu_1,\hdots,\mu_{n+1} \geq 0$ denote the singular values of the differential \smash{$d\tilde{f}_{(x,t)}\colon \big(T_{(x,t)} \tilde{M},\tilde{g}\big) \to (T_{\tilde{f}(x,t)} S^n,g_{S^n})$}, arranged in decreasing order. Since the differential \smash{$d\tilde{f}_{(x,t)}$} has rank at most $n$, it follows that $\mu_{n+1} = 0$.
The eigenvalues of the symmetric bilinear form $\tilde{f}^* g_{S^n}$ with respect to the metric $\tilde{g} = g + r^2 g_{S^1}$ are given by $\mu_1^2,\hdots,\mu_n^2,0$.

Using \eqref{ineq_for_g}, we obtain, on the one hand,
\[
 \tilde{g} = g + r^2 g_{S^1} \geq \rho^2(\Theta) \varphi^* g_{S^{n-1}} + r^2 g_{S^1}.
 \]
On the other hand, the inequality $h^* g_{S^n} \leq g_{S^{n-1}} + 4 g_{S^1}$ implies
\[
 \tilde{f}^* g_{S^n} \leq \varphi^* g_{S^{n-1}} + 4 g_{S^1}.
\]
By assumption, $r > 2\rho(\Theta)$ at the point $(x,t)$. Hence, the min-max characterization of the eigenvalues implies that $\mu_1^2,\hdots,\mu_{n-1}^2 \leq \frac{1}{\rho^2(\Theta)}$ and $\mu_n^2 \leq \frac{4}{r^2}$.
Together with Proposition~\ref{curvature.term}, this implies
\begin{align*}
\big\langle \mathcal{R}^{\tilde{E}} u,u \big\rangle
&\geq -\frac{1}{4} \sum_{\substack{1\leq k,j\leq n\\ j \neq k}} \mu_j \mu_k |u|^2
\geq -\frac{(n-2)(n-1)}{4} \frac{1}{\rho^2(\Theta)} |u|^2 - \frac{n-1}{r} \frac{1}{\rho(\Theta)} |u|^2.
\end{align*}
Recall that $|\nabla \Theta| \leq 1$ by Lemma~\ref{gradient_of_Theta}.
Since $-\psi' > 0$ on $[\theta_-,\theta_+]$, it follows that $|\nabla \Psi| \leq -\psi'(\Theta)$.
Using \eqref{ineq_for_scal_even_dim}, we obtain the pointwise estimate
\begin{align*}
&\frac{1}{4} R |u|^2 + \big\langle \mathcal{R}^{\tilde{E}} u,u \big\rangle + \frac{n(n-1)}{4} \Psi^2 |u|^2 - \frac{n-1}{2} |\nabla \Psi| |u|^2 \\
&\qquad\geq \frac{1}{4} R |u|^2 - \frac{(n-2)(n-1)}{4} \frac{1}{\rho^2(\Theta)} |u|^2 - \frac{n-1}{r} \frac{1}{\rho(\Theta)} |u|^2\\
&\phantom{\qquad\geq }{}+ \frac{n(n-1)}{4} \psi^2(\Theta) |u|^2 + \frac{n-1}{2} \psi'(\Theta) |u|^2 \\
&\qquad\geq -\frac{n-1}{r} \frac{1}{\rho(\Theta)} |u|^2.
\end{align*}
Putting these facts together, we conclude that
\[
\int_{\tilde{M}} \big|\tilde{P} u\big|^2 \leq \frac{n-1}{r} \int_{\tilde{M}} \frac{1}{\rho(\Theta)} |u|^2.
\]
Hence, we can find an element $t_0 \in S^1$ such that
\[
\int_{M \times \{t_0\}} \big|\tilde{P} u\big|^2 \leq \frac{n-1}{r} \int_{M \times \{t_0\}} \frac{1}{\rho(\Theta)} |u|^2
\]
and $\int_{M \times \{t_0\}} \frac{1}{\rho(\Theta)} |u|^2 > 0$.
From this, the assertion follows.
\end{proof}

\begin{Corollary}
\label{limiting_spinor_even_dim}
There exists an element $t_0 \in S^1$ with the following property.
Let $f\colon M \to S^n$ be defined by
\[
 f(x) := \tilde{f}(x,t_0) = h(\varphi(x) ,t_0).
\]
Let $E$ denote the pull-back of $E_0^+$ under the map $f$.
Then there exists a section $s \in C^\infty(M,S \otimes E)$ such that
\[
 \int_M |s|^2 = 1
\]
and
\[
 \nabla_X^{S \otimes E} s + \frac{\rm i}{2} \Psi X \cdot s = 0
\]
for every vector field $X$.
\end{Corollary}

\begin{proof}
Let us consider a sequence $r_\ell \to \infty$. For each $\ell$, Proposition~\ref{existence_of_spinor_even_dim} implies the existence of an element $t_\ell \in S^1$ and a section $u^{(\ell)} \in C^\infty\big(\tilde{M},\tilde{S} \otimes \tilde{E}\big)$ such that
\[
 \int_{M \times \{t_\ell\}} \frac{1}{\rho(\Theta)}\big|u^{(\ell)}\big|^2 = 1
\]
and
\[\int_{M \times \{t_\ell\}} \big|\tilde{P} u^{(\ell)}\big|^2 \leq \frac{n-1}{r_\ell}.
\]
After passing to a subsequence, we may assume that the sequence $t_\ell$ converges to an element $t_0 \in S^1$. We define maps $f\colon M \to S^n$ and $f^{(\ell)}\colon M \to S^n$ by
\[
 f(x) := \tilde{f}(x,t_0) = h(\varphi(x),t_0), \qquad f^{(\ell)}(x) := \tilde{f}(x,t_\ell) = h(\varphi(x) ,t_\ell)
\]
for $x \in M$.
Let $E$ denote the pull-back of $E_0^+$ under $f$, and let $E^{(\ell)}$ denote the pull-back of $E_0^+$ under $f^{(\ell)}$.
The restriction of $u^{(\ell)}$ to $M \times \{t_\ell\}$ gives a section $s^{(\ell)} \in C^\infty\big(M,S \otimes E^{(\ell)}\big)$ such that
\[
 \int_M \frac{1}{\rho(\Theta)}\big|s^{(\ell)}\big|^2 = 1
 \]
and
\begin{equation}
 \int_M \sum_{k=1}^n \left| \nabla_{e_k}^{S \otimes E^{(\ell)}} s^{(\ell)} + \frac{\rm i}{2} \psi(\Theta) e_k \cdot s^{(\ell)} \right|^2 \leq \frac{n-1}{r_\ell}.
\label{eq:almostflat}
\end{equation}
For each $\ell$, we define a bundle map $\sigma^{(\ell)}\colon E^{(\ell)} \to E$ as follows.
For each point $x \in M$, the map \smash{$\sigma^{(\ell)}_x\colon E^{(\ell)}_x = \big(E_0^+\big)_{f^{(\ell)}(x)} \to E_x = \big(E_0^+\big)_{f(x)}$} is defined as the parallel transport along the shortest geodesic from $f^{(\ell)}(x) \in S^n$ to $f(x) \in S^n$.
It is easy to see that $\sigma^{(\ell)}$ is a bundle isometry for each~$\ell$. It follows that the map $\big(\id \otimes \sigma^{(\ell)}\big)\colon S \otimes E^{(\ell)} \to S \otimes E$ is a bundle isometry for each $\ell$.

We may write
\begin{equation}
\label{eq:A_l_term}
\nabla^{S \otimes E}\big(\big(\id \otimes \sigma^{(\ell)}\big) s^{(\ell)}\big) = \big(\id \otimes \sigma^{(\ell)}\big) \big(\nabla^{S \otimes E^{(\ell)}} s^{(\ell)} + A^{(\ell)} s^{(\ell)}\big),
\end{equation}
where $A^{(\ell)}$ is a $1$-form taking values in the endomorphism bundle $\text{\rm End}\big(S \otimes E^{(\ell)}\big)$. Since $t_\ell \to t_0$, the maps $f_\ell$ converge to $f$ smoothly. From this, we deduce that \smash{$\big|A^{(\ell)}\big| \to 0$} uniformly.

By \eqref{eq:almostflat}, \smash{$\nabla^{S \otimes E^{(\ell)}} s^{(\ell)}$} is bounded in $L^2$. Using \eqref{eq:A_l_term}, we conclude that \smash{$\nabla^{S \otimes E} \big(\big(\id \otimes \sigma^{(\ell)}\big) s^{(\ell)}\big)$} is bounded in $L^2$. So, the sequence \smash{$\big(\id \otimes \sigma^{(\ell)}\big) s^{(\ell)} \in C^\infty(M,S \otimes E)$} is bounded in $H^1(M,S \otimes E)$.
After passing to a subsequence, the sequence \smash{$\big(\id \otimes \sigma^{(\ell)}\big) s^{(\ell)} \in C^\infty(M,S \otimes E)$} converges, in the weak topology of $H^1(M,S \otimes E)$, to a section $s$.
Since weak $H^1$-convergence implies strong $L^2$-convergence, the limit $s \in H^1(M,S \otimes E)$ satisfies
\[
\int_M \frac{1}{\rho(\Theta)} |s|^2 = 1.
\]
The inequality \eqref{eq:almostflat} implies that
\begin{equation}
\nabla_X^{S \otimes E} s + \frac{\rm i}{2} \Psi X \cdot s = 0
\label{eq:parallel}
\end{equation}
for every smooth vector field $X$ on $M$, where \eqref{eq:parallel} is understood in the sense of distributions.
Since \eqref{eq:parallel} holds for every smooth vector field $X$ on $M$, it follows
that $s$ is a weak solution of an overdetermined elliptic system.
By elliptic regularity, $s$ is smooth and \eqref{eq:parallel} holds classically.
Rescaling $s$ concludes the proof.
\end{proof}

\begin{Definition}
Let $t_0 \in S^1$, $f\colon M \to S^n$, and $E$ be defined as in Corollary~\ref{limiting_spinor_even_dim}.
We define the modified connection $\nabla^\Psi$ on $S \otimes E $ by
\begin{equation}
\nabla_X^\Psi s = \nabla_X^{S \otimes E} s + \frac{\rm i}{2} \Psi X \cdot s,
\label{def:nablapsi}
\end{equation}
where $\nabla^{S \otimes E}$ denotes the connection on $S \otimes E$ induced by the ones on $S$ and $E$.
\end{Definition}

\begin{Lemma}
The curvature tensor $R^\Psi$ of the connection $\nabla^\Psi$ defined in \eqref{def:nablapsi} satisfies:
\begin{align}
R^\Psi_{X,Y}s
&=
R^{S\otimes E}_{X,Y}s + \frac{\rm i}{2} ({\rm d}\Psi(X)Y-{\rm d}\Psi(Y)X) \cdot s - \frac{1}{4} \Psi^2 (X\cdot Y-Y\cdot X)\cdot s .
\label{eq:RPsi}
\end{align}
Here $R^{S\otimes E}$ denotes the curvature tensor of $\nabla^{S\otimes E}$.
Moreover, the curvature term in the Weitzen\-b\"ock formula satisfies
\begin{align}
\frac{1}{2} \sum_{\substack{1\leq j,k\leq n\\ j \neq k}} e_j\cdot e_k\cdot R^\Psi_{e_j,e_k}s
=
\left(\frac{1}{4} R + \mathcal{R}^E\right)s - \frac{{\rm i}(n-1)}{2} \nabla\Psi\cdot s + \frac{n(n-1)}{4} \Psi^2 s,
\label{eq:Weitzenbock}
\end{align}
where $e_1,\dots,e_n$ is a local orthonormal frame.
\end{Lemma}

\begin{proof}
We check \eqref{eq:RPsi} at a fixed point on $M$. Let $X$ and $Y$ be vector fields defined in a~neighborhood of that point whose covariant derivatives vanish at the point. We compute at that point:
\begin{align*}
\nabla^\Psi_X \nabla^\Psi_Y s
=
\nabla^{S \otimes E}_X \nabla^{S \otimes E}_Y s + \frac{\rm i}{2} {\rm d}\Psi(X)Y\!\cdot\! s + \frac{\rm i}{2} \Psi Y \!\cdot\! \nabla^{S \otimes E}_X s + \frac{\rm i}{2} \Psi X \!\cdot\! \nabla^{S \otimes E}_Y s
 - \frac{1}{4} \Psi^2 X \!\cdot\! Y\cdot s.
\end{align*}
Anti-symmetrizing with respect to $X$ and $Y$ yields \eqref{eq:RPsi}.

As to \eqref{eq:Weitzenbock}, we use formula~(8.8) in \cite[Chapter~II]{Lawson-Michelsohn} and \eqref{eq:RPsi} and we find
\begin{align*}
&\left(\frac{1}{4} R + \mathcal{R}^E\right)s =
\frac{1}{2} \sum_{j \neq k} e_j\cdot e_k\cdot R^{S \otimes E}_{e_j,e_k} s \\
&\qquad=
\frac{1}{2} \sum_{j \neq k} e_j\cdot e_k\cdot \left( R^\Psi_{e_j,e_k} s - \frac{\rm i}{2} ({\rm d}\Psi(e_j)e_k-{\rm d}\Psi(e_k)e_j) \cdot s + \frac{1}{4} \Psi^2 (e_j\cdot e_k-e_k\cdot e_j) \cdot s \right) \\
&\qquad=
\frac{1}{2} \sum_{j \neq k} e_j\cdot e_k\cdot R^\Psi_{e_j,e_k} s + \frac{{\rm i}(n-1)}{2} \nabla\Psi \cdot s - \frac{n(n-1)}{4} \Psi^2 s.\tag*{\qed}
\end{align*} \renewcommand{\qed}{}
\end{proof}

After these preparations, we now complete the proof of Theorem~\ref{spherical_band} for even $n$.
Let $t_0 \in S^1$, $f\colon M \to S^n$, and $E$ be defined as in Corollary~\ref{limiting_spinor_even_dim} and let $\nabla^\Psi$ denote the connection defined in~\eqref{def:nablapsi}.
By Corollary~\ref{limiting_spinor_even_dim}, there exists a section $s \in C^\infty(M,S \otimes E)$ such that $\int_M |s|^2 = 1$ and~${\nabla^\Psi s = 0}$.
Since $M$ is connected and $s$ is $\nabla^\Psi$-parallel, we have that $s \neq 0$ at each point in~$M$. Using \eqref{eq:Weitzenbock} and the fact that $R^\Psi$ annihilates $s$, we find
\begin{align*}
\mathcal{R}^E s
&=
-\frac{1}{4} R s + \frac{{\rm i}(n-1)}{2} \nabla \Psi \cdot s - \frac{n(n-1)}{4} \Psi^2 s \\
&=
-\frac{1}{4} R s + \frac{{\rm i}(n-1)}{2} \psi'(\Theta) \nabla\Theta \cdot s - \frac{n(n-1)}{4} \psi^2(\Theta) s .
\end{align*}
By Lemma~\ref{gradient_of_Theta}, $|\nabla \Theta| \leq 1$. Since $-\psi' > 0$ on $[\theta_-,\theta_+]$, it follows that{\samepage
\begin{align}
\big\langle \mathcal{R}^E s,s \big\rangle
&\leq
-\frac{1}{4} R |s|^2 - \frac{n-1}{2} \psi'(\Theta) |s|^2 - \frac{n(n-1)}{4} \psi^2(\Theta) |s|^2\nonumber \\
&\leq
-\frac{(n-2)(n-1)}{4} \frac{1}{\rho^2(\Theta)} |s|^2. \label{estimate_curvature_E}
\end{align}
In the last step, we have again used \eqref{ineq_for_scal_even_dim}.}

Let us fix an arbitrary point $x \in M$. We consider the singular values $\mu_1, \ldots, \mu_n \geq 0$ of ${\rm d}f_x\colon (T_x M,g) \to (T_{f(x)} S^n,g_{S^n})$, arranged in decreasing order. Since $f$ factors through $S^{n-1}$, the differential ${\rm d}f_x$ has rank at most $n-1$, and we obtain $\mu_n = 0$. The eigenvalues of the symmetric bilinear form $f^* g_{S^n}$ with respect to the metric $g$ are given by $\mu_1^2,\hdots,\mu_{n-1}^2,0$.

Using \eqref{ineq_for_g} together with the inequality $h(\cdot,t_0)^* g_{S^n} \leq g_{S^{n-1}}$, we obtain
\[g \geq \rho^2(\Theta) \varphi^* g_{S^{n-1}} \geq \rho^2(\Theta) f^* g_{S^n}\]
at the point $x$.
Therefore, $\mu_1^2,\hdots,\mu_{n-1}^2 \leq \frac{1}{\rho^2(\Theta)}$.
Using Proposition~\ref{curvature.term} in Appendix~\ref{appendix:Weitzen}, we deduce that
\[
 \big\langle \mathcal{R}^E s ,s \big\rangle \geq -\frac{(n-2)(n-1)}{4} \frac{1}{\rho^2(\Theta)} |s|^2
\]
at the point $x$.
Having established the reverse inequality in \eqref{estimate_curvature_E}, this inequality must be an equality.
In particular all the inequalities in \eqref{estimate_curvature_E} are equalities.
Since $n>2$ and $s \neq 0$ at the point $x$ and $\psi'\neq0$, we can draw the following conclusion:
\begin{itemize}\itemsep=0pt
\item The singular values of ${\rm d}f_x\colon (T_x M,g) \to (T_{f(x)} S^n,g_{S^n})$ are given by $\frac{1}{\rho(\Theta)}, \ldots, \frac{1}{\rho(\Theta)}, 0$. In other words, the eigenvalues of $f^* g_{S^n}$ with respect to $g$ at the point $x$ are given by $\frac{1}{\rho^2(\Theta)}, \ldots, \frac{1}{\rho^2(\Theta)}, 0$.
\item $|\nabla \Theta| = 1$ at the point $x$.
\end{itemize}
Since $|\nabla \Theta| = 1$ at the point $x$, Lemma~\ref{gradient_of_Theta} implies that ${\rm d}\varphi_x(\nabla \Theta) = 0$, hence ${\rm d}f_x(\nabla \Theta) = 0$.
Consequently, $\nabla \Theta$ lies in the nullspace of $f^* g_{S^n}$. Putting these facts together, we obtain
\[f^* g_{S^n} = \frac{1}{\rho^2(\Theta)} (g - {\rm d}\Theta \otimes{\rm d}\Theta),\]
hence
\begin{equation}
\label{a}
g = {\rm d}\Theta \otimes {\rm d}\Theta + \rho^2(\Theta) f^* g_{S^n}
\end{equation}
at the point $x$. On the other hand, using \eqref{ineq_for_g} together with the inequality $h(\cdot,t_0)^* g_{S^n} \leq g_{S^{n-1}}$, we obtain
\begin{equation}
\label{b}
g \geq {\rm d}\Theta \otimes {\rm d}\Theta + \rho^2(\Theta) \varphi^* g_{S^{n-1}} \geq {\rm d}\Theta \otimes {\rm d}\Theta + \rho^2(\Theta) f^* g_{S^n}
\end{equation}
at the point $x$. Combining (\ref{a}) and (\ref{b}), we conclude that
\[g = {\rm d}\Theta \otimes {\rm d}\Theta + \rho^2(\Theta) \varphi^* g_{S^{n-1}}\]
at the point $x$. Since $x$ is arbitrary, we conclude that $g= \Phi^*(g_0)$.
This means that $\Phi$ is a local isometry.
Since the target of $\Phi$ is simply connected and the domain is connected, it follows that~$\Phi$ is a global Riemannian isometry.
The proof of Theorem~\ref{spherical_band} for even $n$ is complete.

\subsection[Proof of Theorem~\ref{spherical_band} for n odd]{Proof of Theorem~\ref{spherical_band} for $\boldsymbol{n}$ odd}

When $n$ is odd, the proof of Theorem~\ref{spherical_band} is simpler, and we just indicate the necessary changes.
Instead of working with $\tilde{M} = M \times S^1$, we work with $M$.
Furthermore, instead of the map $\tilde{f} = h \circ (\varphi \times \id)\colon \tilde{M} \to S^n$, we directly work with $\varphi\colon M \to S^{n-1}$.
Let $S$ denote the spinor bundle over $M$, and let $E_0$ denote the spinor bundle over the round sphere $S^{n-1}$.
Since $n-1$ is even, we may decompose $E_0 = E_0^+ \oplus E_0^-$, where $E_0^+$ and $E_0^-$ denote the $\pm 1$-eigenbundles of the complex volume form.

\begin{Proposition}[cf.\ Cecchini--Zeidler \cite{Cecchini-Zeidler}]
\label{index_odd_dim}
Consider the indices of the following operators:
\begin{itemize}\itemsep=0pt
\item Let $\mathrm{ind}_1$ denote the index of the Dirac operator on $S \otimes \varphi^* E_0^+$ with boundary conditions $s = -{\rm i} \nu \cdot s$ on $\partial_+ M$ and $s = {\rm i} \nu \cdot s$ on $\partial_- M$.
\item Let $\mathrm{ind}_2$ denote the index of the Dirac operator on $S \otimes \varphi^* E_0^+$ with boundary conditions $s = {\rm i} \nu \cdot s$ on $\partial_+ M$ and $s = -{\rm i} \nu \cdot s$ on $\partial_- M$.
\item Let $\mathrm{ind}_3$ denote the index of the Dirac operator on $S \otimes \varphi^* E_0^-$ with boundary conditions $s = -{\rm i} \nu \cdot s$ on $\partial_+ M$ and $s = {\rm i} \nu \cdot s$ on $\partial_- M$.
\item Let $\mathrm{ind}_4$ denote the index of the Dirac operator on $S \otimes \varphi^* E_0^-$ with boundary conditions $s = {\rm i} \nu \cdot s$ on $\partial_+ M$ and $s = -{\rm i} \nu \cdot s$ on $\partial_- M$.
\end{itemize}
Then $\mathrm{ind}_1+\mathrm{ind}_2=0$, $\mathrm{ind}_3+\mathrm{ind}_4=0$, and $\max \{\mathrm{ind}_1,\mathrm{ind}_2, \mathrm{ind}_3, \mathrm{ind}_4\} > 0$.
\end{Proposition}

The proof of Proposition \ref{index_odd_dim} is analogous to the proof of Proposition \ref{index_even_dim} and uses the holographic index theorem.

After switching the bundles $E_0^+$ and $E_0^-$ if necessary, we may assume that ${\max \{\mathrm{ind}_1,\mathrm{ind}_2\} \!> 0}$.
As above, we focus on the case $\mathrm{ind}_1 > 0$.
(The case $\mathrm{ind}_2 > 0$ can be handled analogously.)

Let $E$ denote the pull-back of $E_0^+$ under $\varphi$. Let $\mathcal{D}^{S \otimes E}$ denote the Dirac operator on sections of $S \otimes E$, and let $\mathcal{D}^{\partial M}$ denote the boundary Dirac operator.

Let $\Psi$ be a smooth function on $M$ that will be specified later. If $s$ is section of the bundle~${S \otimes E}$ and $X$ is a vector field on $M$, we define a perturbed covariant derivative of $s$ by the formula
\[
 P_X s = \nabla_{X}^{S \otimes E}s + \frac{\rm i}{2} \Psi X \cdot s.
\]

\begin{Proposition}
\label{integral_formula_odd_dim}
Let $s \in C^\infty(M,S \otimes E)$.
Then
\begin{align*}
&- \int_M \left|\mathcal{D}^{S \otimes E} s - \frac{{\rm i}n}{2} \Psi s\right|^2 + \int_M |Ps|^2 + \frac{1}{4} \int_M R |s|^2 \\
&+ \int_M \big\langle \mathcal{R}^E s,s \big\rangle + \frac{n(n-1)}{4} \int_M \Psi^2 |s|^2 - \frac{{\rm i}(n-1)}{2} \int_M \langle (\nabla \Psi) \cdot s,s \rangle \\
&\qquad= \frac{1}{2} \int_{\partial_+ M} \big\langle \mathcal{D}^{\partial M} s,s + {\rm i} \nu \cdot s \big\rangle + \frac{1}{2} \int_{\partial_+ M} \big\langle s + {\rm i} \nu \cdot s,\mathcal{D}^{\partial M} s \big\rangle \\
&\phantom{\qquad=}{} - \frac{1}{2} \int_{\partial_+ M} (H-(n-1)\Psi) |s|^2 - \frac{n-1}{2} \int_{\partial_+ M} \Psi \langle s + {\rm i} \nu \cdot s,s \rangle \\
&\phantom{\qquad=}{} + \frac{1}{2} \int_{\partial_- M} \big\langle \mathcal{D}^{\partial M} s,s - {\rm i} \nu \cdot s \big\rangle + \frac{1}{2} \int_{\partial_- M} \big\langle s - {\rm i} \nu \cdot s,\mathcal{D}^{\partial M} s \big\rangle \\
&\phantom{\qquad=}{} - \frac{1}{2} \int_{\partial_- M} (H+(n-1)\Psi) |s|^2 + \frac{n-1}{2} \int_{\partial_- M} \Psi \langle s - {\rm i} \nu \cdot s,s \rangle.
\end{align*}
\end{Proposition}

The proof of Proposition~\ref{integral_formula_odd_dim} is analogous to Proposition~\ref{integral_formula_even_dim}.

As above, we define a function $\psi\colon [\theta_-,\theta_+] \to \R$ by $\psi(\theta) = \frac{\rho'(\theta)}{\rho(\theta)}$. The assumption $R \geq R_{g_0} \circ \Phi$ gives
\begin{equation} \label{ineq_for_scal_odd_dim}
R \geq (n-1) \left( -2 \psi'(\Theta) - n \psi^2(\Theta) + \frac{n-2}{\rho^2(\Theta)} \right).
\end{equation}
We define $\Psi\colon M \to \R$ by $\Psi = \psi \circ \Theta$.

At this point, we use the assumption that $\mathrm{ind}_1 > 0$. In view of the deformation invariance of the index, we find a section $s$ of $S \otimes E$ such that
\begin{itemize}\itemsep=0pt
 \item $s$ does not vanish identically,
 \item $D^{S \otimes E } s - \frac{{\rm i}n}{2} \psi s = 0$ on $M$,
 \item $s = -{\rm i} \nu \cdot s$ on $\partial_+ M$ and $s = {\rm i} \nu \cdot s$ on $\partial_- M$.
\end{itemize}
Using Proposition~\ref{integral_formula_odd_dim}, we obtain
\begin{align}
&\int_M |P s|^2 + \frac{1}{4} \int_M R |s|^2
+ \int_M \big\langle \mathcal{R}^E s,s \big\rangle + \frac{n(n-1)}{4} \int_M \Psi^2 |s|^2 - \frac{{\rm i}(n-1)}{2} \int_M \langle (\nabla \Psi) \cdot s,s \rangle\nonumber \\
&\qquad= -\frac{1}{2} \int_{\partial_+ M} (H-(n-1)\Psi) |s|^2 - \frac{1}{2} \int_{\partial_- M} (H+(n-1)\Psi) |s|^2. \label{important_estimate_odd_dim}
\end{align}
By assumption, $H-(n-1)\Psi \geq 0$ on $\partial_+ M$ and $H +(n-1)\Psi \geq 0$ on $\partial_- M$. Therefore, the right-hand side in \eqref{important_estimate_odd_dim} is non-positive.

Fix a point $x \in M $, and let $ \mu_1,\hdots,\mu_n \geq 0$ denote the singular values of the differential ${\rm d}\varphi_x\colon (T_x M,g ) \to \big(T_{f(x)} S^{n-1} ,g_{S^{n-1}}\big)$, arranged in decreasing order. Since the differential ${\rm d}\varphi_x$ has rank at most $n-1$, it follows that $\mu_n = 0$.

Since $\varphi$ is $1$-Lipschitz by assumption, we obtain
\[
 g \geq \rho^2(\Theta) \varphi^* g_{S^{n-1}} .
 \]
Consequently, $\mu_1^2,\hdots,\mu_{n-1}^2 \leq \frac{1}{\rho^2(\Theta)}$. In view of Proposition~\ref{curvature.term}, this implies
\begin{align*}
\big\langle \mathcal{R}^{E} s,s \big\rangle
&\geq -\frac{1}{4} \sum_{\substack{1 \leq j,k \leq n\\ j \neq k}} \mu_j \mu_k |s|^2 \geq -\frac{(n-2)(n-1)}{4} \frac{1}{\rho^2(\Theta)} |s|^2.
\end{align*}
Since $|\nabla \Theta| \leq 1$ and $|\nabla \Psi| \leq -\psi'(\Theta)$, this gives the pointwise estimate
\begin{align*}
&\frac{1}{4} R |s|^2 + \big\langle \mathcal{R}^{E} s,s \big\rangle + \frac{n(n-1)}{4} \Psi^2 |s|^2 - \frac{n-1}{2} |\nabla \Psi| |s|^2 \\
&\qquad\geq \frac{1}{4} R |s|^2 - \frac{(n-2)(n-1)}{4} \frac{1}{\rho^2(\Theta)} |s|^2
+ \frac{n(n-1)}{4} \psi^2(\Theta) |s|^2 + \frac{n-1}{2} \psi'(\Theta) |s|^2 \geq 0.
\end{align*}
In the last step, we have used \eqref{ineq_for_scal_odd_dim}.
Putting these facts together, we conclude that
\[
\int_M | P s|^2 = 0,
\]
hence
\[
 \nabla_X^{S \otimes E} s + \frac{\rm i}{2} \psi(\Theta) X \cdot s = 0
\]
for every vector field $X$.
This is the analogue of Corollary~\ref{limiting_spinor_even_dim}.
From here on, the proof of Theorem~\ref{spherical_band} in the odd-dimensional case proceeds in the same way as in the even-dimensional case.

\section{Proof of Theorem~\ref{punctures}}
\label{proof_punctures}

\subsection[Proof of Theorem~\ref{punctures} for n even]{Proof of Theorem~\ref{punctures} for $\boldsymbol{n}$ even}

We first prove Theorem~\ref{punctures} for even $n$.
The necessary adaptations in the odd-dimensional case will be explained at the end of this section.

Fix an even integer $n>2$.
Consider the warped product metric $g_0 = {\rm d}\theta \otimes {\rm d}\theta + \sin^2 \theta g_{S^{n-1}}$ on $S^{n-1} \times (0,\pi)$.
Let $\Omega$ be a non-compact, connected spin manifold of dimension $n$ without boundary.
Let $g$ be a (possibly incomplete) Riemannian metric on $\Omega$ with scalar curvature~${R \geq n(n-1)}$.
Suppose that $\Phi\colon (\Omega,g) \to \big(S^{n-1} \times (0,\pi),g_0\big)$ is a smooth map satisfying the assumptions of Theorem~\ref{punctures}.

Let $\varphi\colon \Omega \to S^{n-1}$ denote the projection of $\Phi$ to the first factor, and let $\Theta\colon \Omega \to (0,\pi)$ denote the projection of $\Phi$ to the second factor. Since $\Phi$ is proper, it follows that $\Theta$ is proper. Since~${\Phi = (\varphi,\Theta)}$ is $1$-Lipschitz, we obtain
\begin{equation} \label{ineq_for_g_2}
 g \geq {\rm d}\Theta \otimes {\rm d}\Theta + \sin^2 \Theta \varphi^* g_{S^{n-1}}.
\end{equation}
As in Lemma \ref{gradient_of_Theta}, \eqref{ineq_for_g_2} implies that $|\nabla \Theta| \leq 1$, and the inequality is strict unless ${\rm d}\varphi(\nabla \Theta) = 0$.

Throughout this section, we fix a point $z \in S^{n-1} \times \big[\frac{\pi}{3},\frac{2\pi}{3}\big]$ with the property that $z$ is a~regular value of $\Phi$. Since $\Phi$ is proper, $\Phi^{-1}(\{z\})$ is a finite subset of $\Omega$.
Let us fix a real number $\delta_0 \in \big(0,\frac{\pi}{4}\big)$ with the property that the set $\Phi^{-1}(\{z\})$ is contained in a single connected component of the set $\Theta^{-1}([\delta_0,\pi-\delta_0])$.

\begin{Definition}
We denote by $\Delta$ the set of all real numbers $\delta \in (0,\delta_0)$ such that $\delta$ and $\pi-\delta$ are regular values of the function $\Theta\colon \Omega \to (0,\pi)$.
\end{Definition}

By Sard's theorem, $\Delta$ is an open and dense subset of $(0,\delta_0)$.
In the following, we assume that $\delta \in \Delta$. Since $\Theta$ is proper and $\delta$ and $\pi-\delta$ are regular values of $\Theta$, the set $\Theta^{-1}([\delta,\pi-\delta])$ is a compact domain in $\Omega$ with smooth boundary. We denote by $M_\delta$ the connected component of~$\Theta^{-1}([\delta,\pi-\delta])$ that contains the set $\Phi^{-1}(\{z\})$.
Then $M_\delta$ is a compact, connected manifold with boundary.
As above, we may write $\partial M_\delta = \partial_+ M_\delta \cup \partial_- M_\delta$, where
\[\partial_+ M_\delta := \partial M_\delta \cap \Theta^{-1}(\{\pi-\delta\}), \qquad \partial_- M_\delta := \partial M_\delta \cap \Theta^{-1}(\{\delta\}).\]
We first show that the restriction $\Phi|_{M_\delta}\colon M_\delta \to S^{n-1} \times [\delta,\pi-\delta]$ has non-zero degree.

\begin{Lemma}
\label{degree_of_restriction_2}
For each $\delta \in \Delta$, the restriction $\Phi|_{M_\delta}\colon M_\delta \to S^{n-1} \times [\delta,\pi-\delta]$ has the same degree as the map $\Phi\colon \Omega \to S^{n-1} \times (0,\pi)$.
\end{Lemma}

\begin{proof}
The degree of $\Phi$ is obtained by counting the elements of the set $\Phi^{-1}(\{z\})$ with suitable signs.
Since $\Phi^{-1}(\{z\})$ is contained in $M_\delta$, the degree of $\Phi|_{M_\delta}$ coincides with the degree of $\Phi$.
\end{proof}

We consider the product $\tilde{\Omega} = \Omega \times S^1$ equipped with the product metric $\tilde{g} = g + r^2 g_{S^1}$.
Let~$\tilde{M}_\delta = M_\delta \times S^1 \subset \tilde{\Omega}$.
We write $\partial \tilde{M}_\delta = \partial_+ \tilde{M}_\delta \cup \partial_- \tilde{M}_\delta$, where
\[\partial_+ \tilde{M}_\delta := \partial_+ M_\delta \times S^1, \qquad \partial_- \tilde{M}_\delta := \partial_- M_\delta \times S^1.\]
By Lemma~\ref{map_to_Sn}, we can find a smooth map $h\colon S^{n-1} \times S^1 \to S^n$ of degree $\pm 1$ such that $h^* g_{S^n} \leq g_{S^{n-1}} + 4 g_{S^1}$.
We define a smooth map $\tilde{f}\colon \tilde{\Omega} = \Omega \times S^1 \to S^n$ by
$\tilde{f}(x,t) = h(\varphi(x),t)$
for $x \in \Omega$ and $t \in S^1$.

Choose a spin structure on $\Omega$ and let $S$ denote the spinor bundle over $\Omega$. Let $\tilde{S}$ denote the spinor bundle over the product $\tilde{\Omega}$.
Note that $\tilde{S}$ is the pull-back of $S$ under the canonical projection from $\tilde{\Omega} = \Omega \times S^1$ to $\Omega$.

Let $E_0$ denote the spinor bundle of the round sphere $S^n$.
The bundle $E_0$ is equipped with a~natural bundle metric and connection.
Since $n$ is even, we may decompose $E_0$ in the usual way as $E_0 = E_0^+ \oplus E_0^-$, where $E_0^+$ and $E_0^-$ are the eigenbundles of the complex volume form.

For each $\delta \in \Delta$, we define $\mathrm{ind}_1$, $\mathrm{ind}_2$, $\mathrm{ind}_3$, $\mathrm{ind}_4$ as in Proposition~\ref{index_even_dim}, working on $\tilde{M}_\delta$ instead of $\tilde{M}$.
Let
\begin{gather*}
\begin{split}
&\Delta_1 = \{\delta \in \Delta\colon \mathrm{ind}_1 > 0\}, \qquad
\Delta_2 = \{\delta \in \Delta\colon \mathrm{ind}_2 > 0\}, \\
& \Delta_3 = \{\delta \in \Delta\colon \mathrm{ind}_3 > 0\}, \qquad
\Delta_4 = \{\delta \in \Delta\colon \mathrm{ind}_4 > 0\}.
\end{split}
\end{gather*}
By Proposition~\ref{index_even_dim}, we know that
\[
\Delta = \Delta_1 \cup \Delta_2 \cup \Delta_3 \cup \Delta_4.
\]
In particular, at least one of the sets $\Delta_1$, $\Delta_2$, $\Delta_3$, $\Delta_4$ must contain $0$ in its closure.
After switching the bundles $E_0^+$ and $E_0^-$ if necessary, we may assume that one of the sets $\Delta_1$, $\Delta_2$ must contain $0$ in its closure.
In the remainder of this section, we assume that the set $\Delta_1$ contains $0$ in its closure.
(The case when the set $\Delta_2$ contains $0$ in its closure can be handled analogously.)

In the following, we consider a real number $\delta \in \Delta_1$. In other words, $\mathrm{ind}_1$ is positive.
Let~$\tilde{E}$ denote the pull-back of $E_0^+$ under the map $\tilde{f}$.
The bundle metric on $E_0^+$ gives us a bundle metric on $\tilde{E}$.
Moreover, the connection on $E_0^+$ induces a connection on $\tilde{E}$.
As above, we denote by~\smash{$\nabla^{\tilde{S} \otimes \tilde{E}}$} the connection on $\tilde{S} \otimes \tilde{E}$. We denote by~\smash{$\mathcal{D}^{\tilde{S} \otimes \tilde{E}}$} the Dirac operator acting on sections of $\tilde{S} \otimes \tilde{E}$.

\begin{figure}[t]\centering
\begin{overpic}[scale=.75]{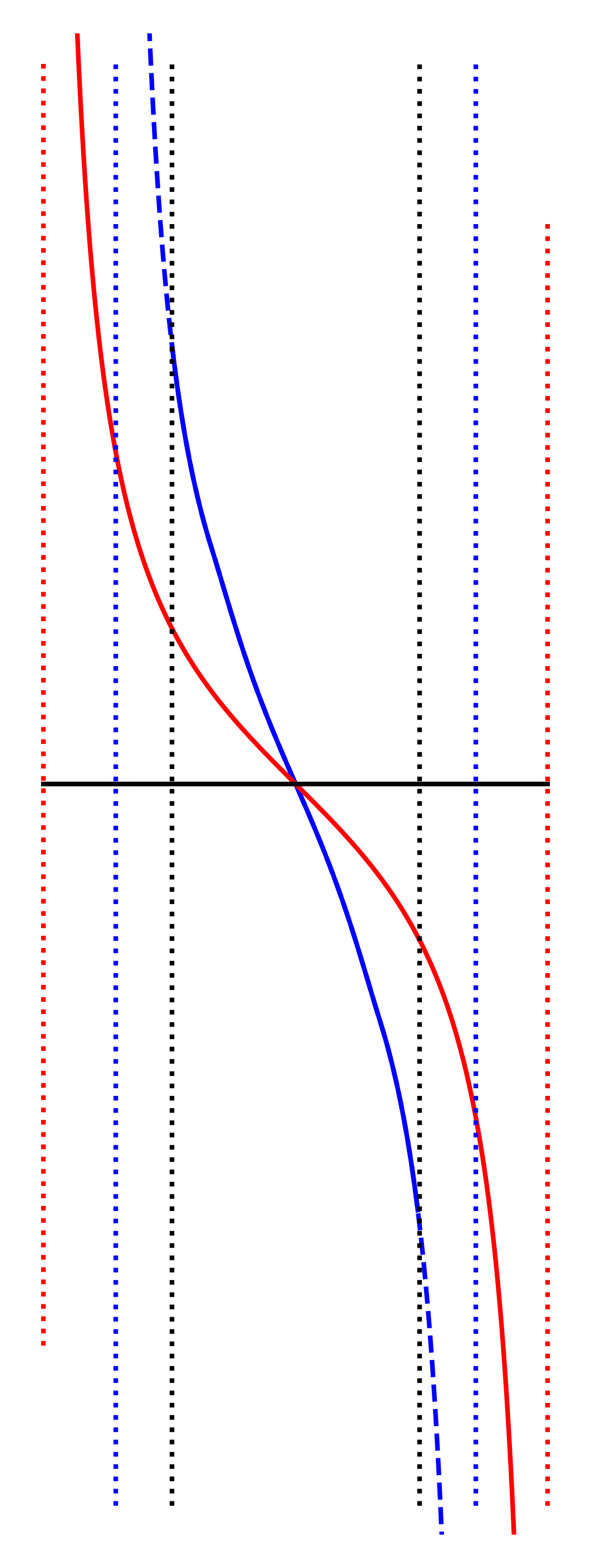}
\put(-4,5){$\textcolor{blue}{\delta-\varepsilon}$}
\put(12,5){$\delta$}
\put(32,90){$\textcolor{blue}{\pi-\delta+\varepsilon}$}
\put(15.5,90){$\pi-\delta$}
\put(-1,14){$\textcolor{red}{0}$}
\put(37,83){$\textcolor{red}{\pi}$}
\put(24,45){$\textcolor{red}{\cot}$}
\put(13,70){$\textcolor{blue}{\psi_{\delta,\varepsilon}}$}
\end{overpic}
\caption{The function $\psi_{\delta,\varepsilon}$.}\label{Fig1}
\end{figure}
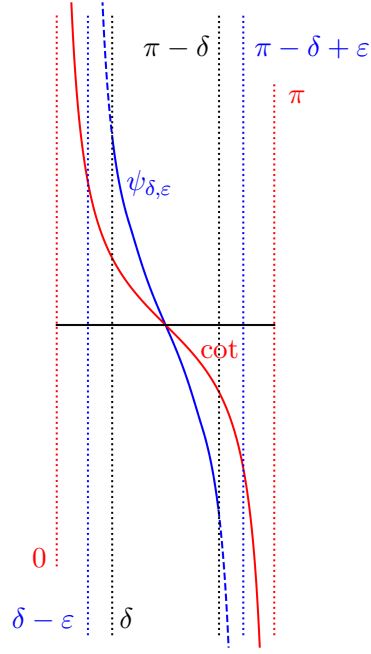

Given $\delta \in \big(0,\frac{\pi}{4}\big)$ and $\varepsilon \in (0,\delta)$, we can find a smooth function $\psi_{\delta,\varepsilon}\colon [\delta,\pi-\delta] \to \R$ with the following properties:
\begin{itemize}\itemsep=0pt
\item $\psi_{\delta,\varepsilon}(\theta) = \cot(\theta-\delta+\varepsilon)$ for $\theta \in \big[\delta,\frac{\pi}{3}\big]$,
\item $\psi_{\delta,\varepsilon}(\theta) = \cot(\theta+\delta-\varepsilon)$ for $\theta \in \big[\frac{2\pi}{3},\pi-\delta\big]$,
\item $|\psi_{\delta,\varepsilon}(\theta) - \cot \theta| \leq K\delta$ for $\theta \in \big[\frac{\pi}{3},\frac{2\pi}{3}\big]$,
\item $\big|\frac{{\rm d}}{{\rm d}\theta}(\psi_{\delta,\varepsilon}(\theta) - \cot \theta)\big| \leq K\delta$ for $\theta \in \big[\frac{\pi}{3},\frac{2\pi}{3}\big]$,
\end{itemize}
see Figure~\ref{Fig1}. Here $K$ is a positive constant that does not depend on $\delta$ and $\varepsilon$.
This choice of the function $\psi_{\delta,\varepsilon}$ is inspired in part by the work of Hirsch, Kazaras, Khuri, and Zhang \cite{Hirsch-Kazaras-Khuri-Zhang}.

Similarly as in Section~\ref{proof_of_band_comparison}, we define a vector field $T$ on $\tilde{M}_\delta = M_\delta \times S^1$ by $T = \frac{1}{r} \frac{\partial}{\partial t}$, where~${t \mapsto (\cos t,\sin t)}$ is the canonical local coordinate on $S^1$.
We define a function ${\Psi_{\delta,\varepsilon}\colon M_\delta\! \to \R}$ by $\Psi_{\delta,\varepsilon} = \psi_{\delta,\varepsilon} \circ \Theta$ and extend $\Psi_{\delta,\varepsilon}$ to a smooth function on $\tilde{M}_\delta$ satisfying $T(\Psi_{\delta,\varepsilon}) = 0$.
Finally, we define an operator $\tilde{P}$ as in \eqref{def:tildeP}, working with $\Psi_{\delta, \varepsilon}$ instead of $\Psi$.

\begin{Proposition}
\label{existence_of_spinor_even_dim_2}
Suppose that $\delta \in \Delta_1$ and suppose that $\varepsilon \in (0,\delta)$ is chosen so that $H \geq -(n-1) \cot \varepsilon$ at each point on the boundary $\partial M_\delta$. Moreover, we assume that $r > 2$. Then we can find an element $t_0 \in S^1$ and a section $u \in C^\infty\big(\tilde{M}_\delta,\tilde{S} \otimes \tilde{E}\big)$ such that
\[\frac{2}{K\delta r} \int_{M_\delta \times \{t_0\}} \frac{1}{\sin \Theta} |u|^2 + \int_{M_\delta \times \{t_0\}} 1_{\{\theta \in [\frac{\pi}{3},\frac{2\pi}{3}]\}} |u|^2 = 1\]
and
\[\int_{M_\delta \times \{t_0\}} \big|\tilde{P} u\big|^2 \leq \frac{n-1}{2} K\delta.\]
\end{Proposition}

\begin{proof}
By assumption, $\delta \in \Delta_1$, so that $\text{ \rm ind}_1 > 0$. In view of the deformation invariance of the index, we can find a section $u \in C^\infty\big(\tilde{M}_\delta,\tilde{S} \otimes \tilde{E}\big)$ such that
\begin{itemize}\itemsep=0pt
\item $u$ does not vanish identically,
\item $\mathcal{D}^{\tilde{S} \otimes \tilde{E}} u - \frac{{\rm i}n}{2} \Psi_{\delta,\varepsilon} u = 0$ on $\tilde{M}_\delta$,
\item $u = -{\rm i} \nu \cdot u$ on $\partial_+ \tilde{M}_\delta$ and $u = {\rm i} \nu \cdot u$ on $\partial_- \tilde{M}_\delta$.
\end{itemize}
As above, it follows from Proposition~\ref{integral_formula_even_dim} that
\begin{align}
&\int_{\tilde{M}_\delta} \big|\tilde{P} u\big|^2 + \int_{\tilde{M}_\delta} \big|\nabla_T^{\tilde{S} \otimes \tilde{E}} u\big|^2 + \frac{1}{4} \int_{\tilde{M}_\delta} R |u|^2 \nonumber\\
&\qquad+ \int_{\tilde{M}_\delta} \big\langle \mathcal{R}^{\tilde{E}} u,u \big\rangle + \frac{n(n-1)}{4} \int_{\tilde{M}_\delta} \Psi_{\delta,\varepsilon}^2 |u|^2 - \frac{{\rm i}(n-1)}{2} \int_{\tilde{M}_\delta} \langle (\nabla \Psi_{\delta,\varepsilon}) \cdot u,u \rangle\nonumber \\
&\phantom{\qquad+}{}=
-\frac{1}{2} \int_{\partial_+ \tilde{M}_\delta} (H-(n-1)\Psi_{\delta,\varepsilon}) |u|^2 - \frac{1}{2} \int_{\partial_- \tilde{M}_\delta} (H+(n-1)\Psi_{\delta,\varepsilon}) |u|^2. \label{important_estimate_even_dim_2}
\end{align}
Note that $\psi_{\delta,\varepsilon}(\pi-\delta) = \cot(\pi-\varepsilon) = -\cot \varepsilon$ and $\psi_{\delta,\varepsilon}(\delta) = \cot \varepsilon$ by our choice of $\psi_{\delta,\varepsilon}$. This implies $H-(n-1)\Psi_{\delta,\varepsilon} = H+(n-1) \cot \varepsilon \geq 0$ on $\partial_+ \tilde{M}_\delta$ and $H+(n-1)\Psi_{\delta,\varepsilon} = H+(n-1) \cot \varepsilon \geq 0$ on $\partial_- \tilde{M}_\delta$. Therefore, the right-hand side in \eqref{important_estimate_even_dim_2} is non-positive.

By assumption, $r > 2 \sin \Theta$. Arguing as in Section~\ref{proof_of_band_comparison}, we can bound
\[
 \big\langle \mathcal{R}^{\tilde{E}} u,u \big\rangle \geq -\frac{(n-2)(n-1)}{4} \frac{1}{\sin^2 \Theta} |u|^2 - \frac{n-1}{r} \frac{1}{\sin \Theta} \, |u|^2.
\]
Moreover, using the inequality $|\nabla \Theta| \leq 1$, we obtain
\[|\nabla \Psi_{\delta,\varepsilon}| \leq |\psi_{\delta,\varepsilon}'(\Theta)| = \frac{1}{\sin^2(\Theta-\delta+\varepsilon)}\]
on the set $\Theta^{-1}\big(\big[\delta,\frac{\pi}{3}\big]\big)$,
\[|\nabla \Psi_{\delta,\varepsilon}| \leq |\psi_{\delta,\varepsilon}'(\Theta)| = \frac{1}{\sin^2(\Theta+\delta-\varepsilon)}\]
on the set $\Theta^{-1}\big(\big[\frac{2\pi}{3},\pi-\delta\big]\big)$, and
\[|\nabla \Psi_{\delta,\varepsilon}| \leq |\psi_{\delta,\varepsilon}'(\Theta)| \leq \frac{1}{\sin^2 \Theta} + K\delta\]
on the set $\Theta^{-1}\big(\big[\frac{\pi}{3},\frac{2\pi}{3}\big]\big)$. Using these facts together with the inequality $R \geq n(n-1)$, we conclude that
\begin{align*}
&\frac{1}{4} R |u|^2 + \big\langle \mathcal{R}^{\tilde{E}} u,u \big\rangle + \frac{n(n-1)}{4} \Psi_{\delta,\varepsilon}^2 |u|^2 - \frac{n-1}{2} |\nabla \Psi_{\delta,\varepsilon}| |u|^2 \\
&\qquad\geq \frac{n(n-1)}{4} |u|^2 - \frac{(n-2)(n-1)}{4} \frac{1}{\sin^2 \Theta} |u|^2 - \frac{n-1}{r} \frac{1}{\sin \Theta} |u|^2 \\
&\phantom{\qquad\geq}{} + \frac{n(n-1)}{4} \cot^2 (\Theta-\delta+\varepsilon) |u|^2 - \frac{n-1}{2} \frac{1}{\sin^2 (\Theta-\delta+\varepsilon)} |u|^2 \\
&\qquad= \frac{(n-2)(n-1)}{4} \left( \frac{1}{\sin^2 (\Theta-\delta+\varepsilon)} - \frac{1}{\sin^2 \Theta} \right) |u|^2 - \frac{n-1}{r} \frac{1}{\sin \Theta} |u|^2 \\
&\qquad\geq -\frac{n-1}{r} \frac{1}{\sin \Theta} |u|^2
\end{align*}
on the set $\Theta^{-1}\big(\big[\delta,\frac{\pi}{3}\big]\big)$,
\begin{align*}
&\frac{1}{4} R |u|^2 + \big\langle \mathcal{R}^{\tilde{E}} u,u \big\rangle + \frac{n(n-1)}{4} \Psi_{\delta,\varepsilon}^2 |u|^2 - \frac{n-1}{2} |\nabla \Psi_{\delta,\varepsilon}| |u|^2 \\
&\qquad\geq \frac{n(n-1)}{4} |u|^2 - \frac{(n-2)(n-1)}{4} \frac{1}{\sin^2 \Theta} |u|^2 - \frac{n-1}{r} \frac{1}{\sin \Theta} |u|^2 \\
&\phantom{\qquad\geq}{} + \frac{n(n-1)}{4} \cot^2 (\Theta+\delta-\varepsilon) |u|^2 - \frac{n-1}{2} \frac{1}{\sin^2 (\Theta+\delta-\varepsilon)} |u|^2 \\
&\qquad= \frac{(n-2)(n-1)}{4} \left( \frac{1}{\sin^2 (\Theta+\delta-\varepsilon)} - \frac{1}{\sin^2 \Theta} \right) |u|^2 - \frac{n-1}{r} \frac{1}{\sin \Theta} |u|^2 \\
&\qquad\geq -\frac{n-1}{r} \frac{1}{\sin \Theta} |u|^2
\end{align*}
on the set $\Theta^{-1}\big(\big[\frac{2\pi}{3},\pi-\delta\big]\big)$, and
\begin{align*}
&\frac{1}{4} R |u|^2 + \big\langle \mathcal{R}^{\tilde{E}} u,u \big\rangle + \frac{n(n-1)}{4} \Psi_{\delta,\varepsilon}^2 |u|^2 - \frac{n-1}{2} |\nabla \Psi_{\delta,\varepsilon}| |u|^2 \\
&\qquad\geq \frac{n(n-1)}{4} |u|^2 - \frac{(n-2)(n-1)}{4} \frac{1}{\sin^2 \Theta} |u|^2 - \frac{n-1}{r} \frac{1}{\sin \Theta} |u|^2 \\
&\phantom{\qquad\geq}{} + \frac{n(n-1)}{4} \cot^2 \Theta |u|^2 - \frac{n-1}{2} \frac{1}{\sin^2 \Theta} |u|^2 - \frac{n-1}{2} K\delta |u|^2 \\
&\qquad= -\frac{n-1}{r} \frac{1}{\sin \Theta} |u|^2 - \frac{n-1}{2} K\delta |u|^2
\end{align*}
on the set $\Theta^{-1}\big(\big[\frac{\pi}{3},\frac{2\pi}{3}\big]\big)$.
Putting these facts together, we obtain
\[
 \int_{\tilde{M}_\delta} \big|\tilde{P} u\big|^2 \leq \frac{n-1}{r} \int_{\tilde{M}_\delta} \frac{1}{\sin \Theta} |u|^2 + \frac{n-1}{2} K\delta \int_{\tilde{M}_\delta} 1_{\Theta^{-1}([\frac{\pi}{3},\frac{2\pi}{3}])} |u|^2.
\]
Hence, we can find an element $t_0 \in S^1$ such that
\[
 \int_{M_\delta \times \{t_0\}} \big|\tilde{P} u\big|^2 \leq \frac{n-1}{r} \int_{M_\delta \times \{t_0\}} \frac{1}{\sin \Theta} |u|^2 + \frac{n-1}{2} K\delta \int_{M_\delta \times \{t_0\}} 1_{\Theta^{-1}([\frac{\pi}{3},\frac{2\pi}{3}])} |u|^2
\]
and
\[
 \frac{2}{K\delta r} \int_{M_\delta \times \{t_0\}} \frac{1}{\sin \Theta} |u|^2 + \int_{M_\delta \times \{t_0\}} 1_{\Theta^{-1}([\frac{\pi}{3},\frac{2\pi}{3}])} |u|^2 > 0.
\]
From this, the assertion follows.
\end{proof}

\begin{Corollary}
\label{limiting_spinor_even_dim_2}
Suppose that $\delta \in \Delta_1$ and suppose that $\varepsilon \in (0,\delta)$ is chosen so that $H \geq -{(n-1)} \cot \varepsilon$ at each point on $\partial M_\delta$. Then there exists an element $t_0 \in S^1$ with the following property.
Let $f\colon \Omega \to S^n$ be defined by $f(x) := \tilde{f}(x,t_0) = h(\varphi(x),t_0)$ for $x \in \Omega$. Let $E$ denote the pull-back of $E_0^+$ under the map $f$.
Then there exists a section $s \in H^1(M_\delta,S \otimes E)$ such that
\[\int_{M_\delta} 1_{\Theta^{-1}([\frac{\pi}{3},\frac{2\pi}{3}])} |s|^2 = 1\]
and
\[\int_{M_\delta} \sum_{k=1}^n \left| \nabla_{e_k}^{S \otimes E} s + \frac{\rm i}{2} \psi_{\delta,\varepsilon}(\Theta) e_k \cdot s \right|^2 \leq \frac{n-1}{2} K\delta.\]
Here, $\nabla^{S \otimes E}$ denotes the connection on $S \otimes E$.
\end{Corollary}

\begin{proof}
Let us consider a sequence $r_\ell \to \infty$.
For each $\ell$, we can find an element $t_\ell \in S^1$ and a~section $u^{(\ell)} \in C^\infty\big(\tilde{M}_\delta,\tilde{S} \otimes \tilde{E}\big)$ such that
\[\frac{2}{K\delta r_\ell} \int_{M_\delta \times \{t_\ell\}} \frac{1}{\sin \Theta} \big|u^{(\ell)}\big|^2 + \int_{M_\delta \times \{t_\ell\}} 1_{\Theta^{-1}([\frac{\pi}{3},\frac{2\pi}{3}])} \big|u^{(\ell)}\big|^2 = 1\]
and
\[\int_{M_\delta \times \{t_\ell\}} \big|\tilde{P} u^{(\ell)}\big|^2 \leq \frac{n-1}{2} K\delta.\]
After passing to a subsequence, we may assume that the sequence $t_\ell$ converges to an element $t_0 \in S^1$. As in Section \ref{proof_of_band_comparison}, we define maps $f\colon \Omega \to S^n$ and $f^{(\ell)}\colon \Omega \to S^n$ by
\[
 f(x) := \tilde{f}(x,t_0) = h(\varphi(x),t_0), \qquad f^{(\ell)}(x) := \tilde{f}(x,t_\ell) = h(\varphi(x) ,t_\ell)
\]
for $x \in \Omega$.
Let $E$ denote the pull-back of $E_0^+$ under $f$, and let $E^{(\ell)}$ denote the pull-back of $E_0^+$ under the map $f^{(\ell)}$. The restriction of $u^{(\ell)}$ to $M_\delta \times \{t_\ell\}$ gives a section $s^{(\ell)} \in C^\infty\big(M_\delta,S \otimes E^{(\ell)}\big)$ such that
\[\frac{2}{K\delta r_\ell} \int_{M_\delta} \frac{1}{\sin \Theta} \big|s^{(\ell)}\big|^2 + \int_{M_\delta} 1_{\Theta^{-1}([\frac{\pi}{3},\frac{2\pi}{3}])} \big|s^{(\ell)}\big|^2 = 1\]
and
\[\int_{M_\delta} \sum_{k=1}^n \left| \nabla_{e_k}^{S \otimes E^{(\ell)}} s^{(\ell)} + \frac{\rm i}{2} \psi_{\delta,\varepsilon}(\Theta) e_k \cdot s^{(l)} \right|^2 \leq \frac{n-1}{2} K\delta.\]
Since $M_\delta$ is connected, we may estimate
\begin{align*}
\begin{aligned}
\int_{M_\delta} \big|s^{(\ell)}\big|^2
\leq{}& C(\delta,\varepsilon) \int_{M_\delta} \sum_{k=1}^n \left| \nabla_{e_k}^{S \otimes E^{(\ell)}} s^{(\ell)} + \frac{\rm i}{2} \psi_{\delta,\varepsilon}(\Theta) e_k \cdot s^{(l)} \right|^2 \\
& + C(\delta,\varepsilon) \int_{M_\delta} 1_{\Theta^{-1}([\frac{\pi}{3},\frac{2\pi}{3}])} \big|s^{(\ell)}\big|^2
\end{aligned}
\end{align*}
(see \cite{Baer-Brendle-Chow-Hanke}).
This implies \smash{$\int_{M_\delta} \big|s^{(\ell)}\big|^2 \leq C(\delta,\varepsilon)$}.
Analogously to Section~\ref{proof_of_band_comparison}, we consider a sequence of bundle maps $\sigma^{(\ell)}\colon E^{(\ell)} \to E$. After passing to a~subsequence, the sequence $ \big(\id \otimes \sigma^{(\ell)}\big) s^{(\ell)} \in C^\infty(M_\delta,S \otimes E)$ converges, in the weak topology of $H^1(M_\delta,S \otimes E)$, to a section~$s$. The section $s \in H^1(M_\delta,S \otimes E)$ has all the desired properties.
\end{proof}

\begin{Corollary}
\label{special_spinor_on_Omega}
There exists an element $t_0 \in S^1$ with the following property.
Let $f\colon \Omega \to S^n$ be defined by $f(x) := \tilde{f}(x,t_0) = h(\varphi(x),t_0)$ for $x \in \Omega$.
Let $E$ denote the pull-back of $E_0^+$ under the map $f$.
Then there exists a section $s \in C^\infty(\Omega,S \otimes E)$ such that
\[\int_\Omega 1_{\Theta^{-1}([\frac{\pi}{3},\frac{2\pi}{3}])} |s|^2 = 1\]
and
\[\nabla_X^{S \otimes E} s + \frac{\rm i}{2} \cot \Theta X \cdot s = 0\]
for every vector field $X$.
Here, $\nabla^{S \otimes E}$ denotes the connection on $S \otimes E$.
\end{Corollary}

\begin{proof}
Recall that we are assuming that the set $\Delta_1$ contains $0$ in its closure. In other words, we can find a sequence of real numbers $\delta_\ell \in \Delta_1$ converging to $0$.
After passing to a subsequence, we may assume that the sequence $\delta_\ell$ is monotonically decreasing.
Consequently, $M_{\delta_l}$ is an increasing sequence of compact domains in $\Omega$, and $\bigcup_l M_{\delta_l} = \Omega$.

We choose a sequence $\varepsilon_\ell \in (0,\delta_\ell)$ such that $H \geq -(n-1) \cot \varepsilon_\ell$ at each point on $\partial M_{\delta_\ell}$.
By Corollary~\ref{limiting_spinor_even_dim_2}, we can find a sequence of elements $t_\ell \in S^1$ with the following property.
Let~${f^{(\ell)}\colon \Omega \to S^n}$ be defined by $f^{(\ell)}(x) := \tilde{f}(x,t_\ell) = h(\varphi(x),t_\ell)$ for $x \in \Omega$.
Let $E^{(\ell)}$ denote the pull-back of $E_0^+$ under the map $f^{(\ell)}$.
Then there exists a section $s^{(\ell)} \in H^1\big(M_{\delta_\ell},S \otimes E^{(\ell)}\big)$ such that
\[
 \int_{M_{\delta_\ell}} 1_{\Theta^{-1}([\frac{\pi}{3},\frac{2\pi}{3}])} \big|s^{(\ell)}\big|^2 = 1
 \]
and
\[
 \int_{M_{\delta_\ell}} \sum_{k=1}^n \left| \nabla_{e_k}^{S \otimes E^{(\ell)}} s^{(\ell)} + \frac{\rm i}{2} \psi_{\delta_\ell,\varepsilon_\ell}(\Theta) e_k \cdot s^{(\ell)} \right|^2 \leq \frac{n-1}{2} K\delta_\ell.
 \]
Since $\Omega$ is connected, results in \cite{Baer-Brendle-Chow-Hanke} imply that the sequence $s^{(\ell)}$ is bounded in $H_{\mathrm{loc}}^1$.

After passing to a subsequence, we may assume that the sequence $t_\ell$ converges to an element~${t_0 \in S^1}$. We define a map $f\colon \Omega \to S^n$ by $f(x) := \tilde{f}(x,t_0) = h(\varphi(x),t_0)$ for $x \in \Omega$.
Let $E$ denote the pull-back of $E_0^+$ under $f$.
As in Section \ref{proof_of_band_comparison}, we consider a sequence of bundle maps $\sigma^{(\ell)}\colon E^{(\ell)} \to E$. After passing to a subsequence, the sequence $\big(\id \otimes \sigma^{(\ell)}\big) s^{(\ell)} \in H^1(M_{\delta_\ell},S \otimes E)$ converges weakly in $H_{\text{\rm loc}}^1$ to a section $s \in H_{\mathrm{loc}}^1(\Omega,S \otimes E)$. The section $s$ satisfies
\[\int_\Omega 1_{\Theta^{-1}([\frac{\pi}{3},\frac{2\pi}{3}])} |s|^2 = 1\]
and
\begin{equation}
\nabla_X^{S \otimes E} s + \frac{\rm i}{2} \cot \Theta X \cdot s = 0
\label{eq:covariant_derivative_s}
\end{equation}
for every vector field $X$, where \eqref{eq:covariant_derivative_s} is understood in the sense of distributions. Since $s$ is a weak solution of an overdetermined elliptic system, we conclude that $s$ is smooth and \eqref{eq:covariant_derivative_s} holds in the classical sense.
\end{proof}

Having established Corollary~\ref{special_spinor_on_Omega}, the proof of Theorem~\ref{punctures} now proceeds as in Section~\ref{proof_of_band_comparison}, with the choice $\rho(\theta) = \sin \theta$ and $\psi(\theta) = \cot \theta$.
As in Section~\ref{proof_of_band_comparison}, we conclude that $g = \Phi^*(g_0)$.
In other words, $\Phi$ is a local Riemannian isometry.
Since $\Phi$ is proper, the domain is connected, and the target is simply connected, $\Phi$ is a global Riemannian isometry.
This completes the proof of Theorem~\ref{punctures} for $n$ even.

\subsection[Proof of Theorem~\ref{punctures} for n odd]{Proof of Theorem~\ref{punctures} for $\boldsymbol{n}$ odd}

When $n$ is odd, the proof of Theorem~\ref{punctures} is simpler, as we can work with $\Omega$ and $M_\delta$ directly, and we do not need to consider the Cartesian product with $S^1$. We omit the details.

\appendix

\appendix
\section{The curvature term in the Weitzenb\"ock formula}\label{appendix:Weitzen}

In this section, we recall a well-known estimate for the curvature term in the Weitzenb\"ock formula. Let us fix integers $n,N \geq 2$. Let $(M,g)$ be a Riemannian spin manifold of dimension~$n$, and let $f\colon (M,g) \to \big(S^N,g_{S^N}\big)$ be a smooth map to the round unit sphere of dimension $N$. Let~${S \to M}$ denote the spinor bundle of $M$, let $E_0 \to S^N$ denote the spinor bundle of $S^N$ and set $E = f^* E_0$. Let $\mathcal{R}^E$ denote the curvature term appearing in the Weitzenb\"ock formula for the square of the twisted Dirac operator on $S \otimes E$, so that
\[
\big(\mathcal{D}^{S \otimes E }\big)^2 s = \big(\nabla^{\tilde{S} \otimes \tilde{E}}\big)^*\nabla^{\tilde{S} \otimes \tilde{E}} + \frac{1}{4} R_g s + \mathcal{R}^{E} s.
\]
The following estimate for the curvature term $\mathcal{R}^E$ is well known.

\begin{Proposition}[cf.\ Llarull \cite{Llarull}]
\label{curvature.term}
Let $x \in M$ and let $ \mu_1,\hdots,\mu_n \geq 0$ denote the singular values of the differential ${\rm d}f_x\colon (T_x M , g_x) \to \big(T_{f(x)} S^N,g_{S^N}\big)$. Then
\[
 \big|\mathcal{R}^E s\big| \leq \frac{1}{4} \sum_{\substack{1 \leq j,k \leq n\\ j \neq k}} \mu_j \mu_k |s|
\]
for all $s \in S_x \otimes E_x$.
\end{Proposition}

\begin{proof}
Let $m$ denote the rank of the differential ${\rm d}f_x$. Clearly, $m \leq \min \{n,N\}$. We assume that the singular values are arranged so that $\mu_k > 0$ for $1 \leq k \leq m$ and $\mu_k = 0$ for $m+1 \leq k \leq N$. We can find an orthonormal basis $\{e_1,\hdots,e_n\}$ of $(T_x M, g_x)$ and an orthonormal basis $\{\varepsilon_1, \ldots, \varepsilon_N\}$ of $\big(T_{f(x)} S^N, g_{S^N}\big)$ such that ${\rm d}f_x(e_k) = \mu_k \varepsilon_k$ for $1 \leq k \leq m$ and ${\rm d}f_x(e_k) = 0$ for $m+1 \leq k \leq n$.

Let $F^{E_0} \in \Omega^2(S^N, {\rm End}(E_0))$ denote the curvature of $E_0$, and let $F^E \in \Omega^2(M, {\rm End}(E))$ denote the curvature of $E$. For each $s \in S_x \otimes E_x = S_x \otimes (E_0)_{f(x)}$, formula~(8.22) in \cite[Chapter II]{Lawson-Michelsohn} gives
\begin{align*}
\mathcal{R}^E s
&= \frac{1}{2} \sum_{\substack{1 \leq j,k \leq n\\ j \neq k}} \big((e_j \cdot e_k) \otimes F^E_{e_j, e_k}\big) \cdot s = \frac{1}{2} \sum_{\substack{1 \leq j,k \leq n\\ j \neq k}} \big((e_j \cdot e_k) \otimes F^{E_0}_{{\rm d}f_x(e_j), {\rm d}f_x(e_k)}\big) \cdot s \\
&= \frac{1}{2} \sum_{\substack{1 \leq j,k \leq m\\ j \neq k}} \mu_j \mu_k \big((e_j \cdot e_k) \otimes F^{E_0}_{\varepsilon_j,\varepsilon_k}\big) \cdot s.
\end{align*}
Since the curvature operator of $S^N$ acts as the identity on $2$-forms, we obtain by formula~(4.37) in \cite[Chapter II]{Lawson-Michelsohn} (also compare \cite[Lemma 4.3]{Llarull}) that $F^{E_0}_{\varepsilon_j,\varepsilon_k} \eta = \frac{1}{2} \varepsilon_k \cdot \varepsilon_j \cdot \eta$ for all $j \neq k$ and all $\eta \in (E_0)_{f(x)}$.
Putting everything together, it follows that
\[
 \mathcal{R}^E s = \frac{1}{4} \sum_{\substack{1 \leq j,k \leq m\\ j \neq k}} \mu_j \mu_k ((e_j \cdot e_k) \otimes (\varepsilon_k \cdot \varepsilon_j)) \cdot s.
\]
For each pair $j \neq k$, Clifford multiplication by $(e_j \cdot e_k) \otimes (\varepsilon_k \cdot \varepsilon_j)$ on $S_x \otimes E_x = S_x \otimes (E_0)_{f(x)}$ is a self-adjoint involution, hence an isometry.
Therefore, $
 |((e_j \cdot e_k) \otimes (\varepsilon_k \cdot \varepsilon_j)) \cdot s| = |s|$ for all $1 \leq j,k \leq m$.
This finally implies
\begin{equation*}
\big|\mathcal{R}^E s\big| \leq \frac{1}{4} \sum_{\substack{1 \leq j,k \leq m\\ j \neq k}} \mu_j \mu_k |s|.\tag*{\qed}
\end{equation*} \renewcommand{\qed}{}
\end{proof}

\section{A holographic index theorem}

\label{holographic}

We prove a theorem which relates the index on a manifold with boundary with that of the boundary.
We only need it for twisted spinorial Dirac operators but since it may be of independent interest, we show it for the larger class of self-adjoint Dirac-type operators.

Let $M$ be a compact Riemannian manifold with boundary $\partial M$ with outward unit normal vector field $\nu$.
Let $S \to M$ be a Hermitian vector bundle.
Let $\mathcal{D}$ be a differential operator of first order.
Its principal symbol is characterized by $\mathcal{D}(fs)=f\mathcal{D}s+\sigma_{\mathcal{D}}({\rm d}f)s$.
We say that $\mathcal{D}$ is of \emph{Dirac type} if its principal symbol satisfies the Clifford relations
\[
\sigma_{\mathcal{D}}(\xi)\sigma_{\mathcal{D}}(\eta) + \sigma_{\mathcal{D}}(\eta)\sigma_{\mathcal{D}}(\xi) = -2g(\xi,\eta)
\]
for all $\xi,\eta\in T^*_xM$ and $x\in M$.
In particular, if $\mathcal{D}$ is of Dirac type, then $\mathcal{D}$ is elliptic.
All generalized Dirac operators in the sense of Gromov and Lawson are of Dirac type.

We assume that the restriction of $S$ to the boundary $\partial M$ splits into two orthogonal subbundles, $S|_{\partial M} = S^+\oplus S^-$.
A first-order differential operator $\mathcal{D}^\partial\colon C^\infty(\partial M, S) \to C^\infty(\partial M, S)$ is called \emph{adapted to} $\mathcal{D}$ if the principal symbols are related by
\begin{equation}
\sigma_{\mathcal{D}^\partial}(\xi)=-\sigma_{\mathcal{D}}\big(\nu^\flat\big)^{-1}\sigma_{\mathcal{D}}(\xi)
\label{eq:boundarycompatibility}
\end{equation}
for all $\xi\in T^*\partial M$.
Here $\nu^\flat$ is the $1$-form metrically related to $\nu$.\footnote{In \cite{Baer-Ballmann} and many other references one works with the inward unit normal rather than the outward pointing one. Then there is no $-$ sign in \eqref{eq:boundarycompatibility} and the roles of $V_<$ and $V_>$ in the proof of Theorem~\ref{thm:HolIndex} get interchanged.}
If $\mathcal{D}$ is of Dirac type then so is $\mathcal{D}^\partial$.

If $\mathcal{D}^\partial$ interchanges the subbundles, i.e., $\mathcal{D}^\partial\colon C^\infty\big(\partial M, S^\pm\big) \to C^\infty\big(\partial M, S^\mp\big)$, then we call $\mathcal{D}^\partial$ an \emph{odd operator}.
We denote by $C^\infty_\pm(M,S)$ the space of all sections $s$ of $S$ which are smooth up to the boundary and satisfy $s(x)\in S^\pm_x$ for all $x\in\partial M$.

\begin{Theorem}\label{thm:HolIndex}
Let $\mathcal{D}$ be a formally self-adjoint Dirac-type operator and let $\mathcal{D}^\partial$ be a formally self-adjoint odd operator adapted to $\mathcal{D}$.
Assume that $\sigma_{\mathcal{D}}\big(\nu^\flat\big)$ preserves the bundles $S^+$ and $S^-$ and anti-commutes with $\mathcal{D}^\partial$.
Then the operators $\mathcal{D}\colon C^\infty_\pm(M,S)\to C^\infty(M,S)$ and $\mathcal{D}^\partial\colon C^\infty\big(\partial M,S^\pm\big)\to C^\infty\big(\partial M,S^\mp\big)$ are Fredholm and their indices satisfy
\begin{align}
&\mathrm{ind}(\mathcal{D}\colon C^\infty_+(M,S) \to C^\infty(M,S)) \nonumber \\
&\qquad=
-\mathrm{ind}(\mathcal{D}\colon C^\infty_-(M,S)\to C^\infty(M,S)) \label{eq:ind1}\\
&\qquad=
\frac{1}{2} \mathrm{ind}\big(\mathcal{D}^\partial\colon C^\infty\big(\partial M,S^+\big)\to C^\infty(\partial M,S^-)\big) \label{eq:ind2}\\
&\qquad=
-\frac{1}{2} \mathrm{ind}\big(\mathcal{D}^\partial\colon C^\infty(\partial M,S^-)\to C^\infty\big(\partial M,S^+\big)\big) .\label{eq:ind3}
\end{align}
\end{Theorem}

\begin{proof}
The Fredholm property of $\mathcal{D}\colon C^\infty_\pm(M,S)\to C^\infty(M,S)$ follows from Corollary~7.23, Proposition~7.24, and Theorem~7.17 combined with Corollary~8.6 in \cite{Baer-Ballmann}.
Note that the completeness and coercivity at infinity required in \cite{Baer-Ballmann} is automatic in our situation as $M$ is compact.

Since $\sigma_{\mathcal{D}}\big(\nu^\flat\big)$ preserves the bundles $S^+$ and $S^-$, the boundary conditions $s\in S^+$ and $s\in S^-$ are adjoint to each other.
Corollary~8.6 in \cite{Baer-Ballmann} implies
\begin{align*}
\mathrm{ind}(\mathcal{D}\colon C^\infty_\pm(M,S)\to C^\infty(M,S))
={}&
\dim\ker(\mathcal{D}\colon C^\infty_\pm(M,S)\to C^\infty(M,S)) \\
& -
\dim\ker(\mathcal{D}\colon C^\infty_\mp(M,S)\to C^\infty(M,S)).
\end{align*}
In particular, this proves \eqref{eq:ind1}.

Elliptic operators on compact manifolds without boundary are always Fredholm.
Since $\mathcal{D}^\partial$ is elliptic and formally self-adjoint, its $L^2$-closure is self-adjoint.
Since the $L^2$-closures of the restrictions $C^\infty\big(\partial M,S^\pm\big)\to C^\infty\big(\partial M,S^\mp\big)$ are adjoints of each other, we obtain \eqref{eq:ind3}.
It remains to prove \eqref{eq:ind2}.

Let $V_>$ and $V_<$ denote the subspaces of the Sobolev space $H^{\frac{1}{2}}(\partial M,S)$ spanned by the eigenspaces of $\mathcal{D}^\partial$ corresponding to the positive or negative eigenvalues, respectively.
Let $H\subset C^\infty(\partial M,S)$ denote the kernel of $\mathcal{D}^\partial$.
Since $\mathcal{D}^\partial$ interchanges $S^+$ and $S^-$, we may decompose $H=H^+\oplus H^-$, where $H^\pm := H\cap C^\infty\big(\partial M,S^\pm\big)$.
This gives an $L^2$-orthogonal decomposition
\[
H^{\frac{1}{2}}(\partial M,S) = V_> \oplus H^+ \oplus H^- \oplus V_< .
\]
Let $\sigma$ denote the self-adjoint bundle involution on $S$ with the property that $S^\pm$ are the eigen\-spaces to the eigenvalues $\pm 1$.
Since $\mathcal{D}^\partial$ interchanges $S^+$ and $S^-$, it anti-commutes with $\sigma$.
Thus, $\sigma$~maps the eigenspace of $\mathcal{D}^\partial$ for the eigenvalue $\lambda$ isomorphically onto that of $-\lambda$.

Now any $s\in V_>\oplus V_<$ can be expanded into eigensections, $s=\sum_{\lambda\neq0}s_\lambda$.
If furthermore \smash{$s\in H^{\frac{1}{2}}\big(\partial M,S^+\big)$}, then
\[
\sum_{\lambda\neq0}s_\lambda
=
s
=
\sigma s
=
\sum_{\lambda\neq0}\sigma s_\lambda
\]
and therefore $
\sigma s_\lambda = s_{-\lambda}$. Thus, $(V_>\oplus V_<) \cap H^{\frac{1}{2}}\big(\partial M,S^+\big)$ is the graph of the map $\sigma\colon V_> \to V_<$.
We introduce a deformation parameter $t\in [0,1]$ and consider the graph of $t\sigma$.
More precisely, we put $B_t := H^+ \oplus \{s+t\sigma s \colon s \in V_>\}$.
Each $B_t$ is an $\infty$-regular elliptic boundary condition for~$\mathcal{D}$ in the sense of \cite{Baer-Ballmann}.
By deformation invariance of the index, $\mathrm{ind}(\mathcal{D},B_1) = \mathrm{ind}(\mathcal{D},B_0)$. In other words, $\mathrm{ind}(\mathcal{D}\colon C^\infty_+(M,S)\to C^\infty(M,S))$ coincides with the index of $\mathcal{D}$ subject to the boundary condition $H^+\oplus V_>$.

We observe that $H^+\oplus V_>$ is a finite-dimensional modification of the Atiyah--Patodi--Singer boundary condition $V_>$.
Since $\sigma\big(\nu^\flat\big)$ anti-commutes with $\mathcal{D}^\partial$, the adjoint boundary condition of $V_>$ is \smash{$\big(\sigma\big(\nu^\flat\big)V_>\big)^\perp = V_<^\perp=H\oplus V_>$} where $\perp$ denotes the $L^2$-orthogonal complement in $H^{\frac{1}{2}}(\partial M,S)$. Therefore,
\begin{equation} \label{eq:index_a}
\mathrm{ind}(\mathcal{D},V_>) = -\mathrm{ind}(\mathcal{D},H\oplus V_>).
\end{equation}
By \cite[Corollary~8.8]{Baer-Ballmann},
\begin{equation} \label{eq:index_b}
\mathrm{ind}(\mathcal{D},H\oplus V_>) = \mathrm{ind}(\mathcal{D},V_>)+\dim(H).
\end{equation}
Combining \eqref{eq:index_a} and \eqref{eq:index_b} gives
\[
\mathrm{ind}(\mathcal{D},V_>)
=
-\frac{1}{2} \dim(H).
\]
Using \cite[Corollary~8.8]{Baer-Ballmann} again, we obtain
\begin{align*}
\mathrm{ind}(\mathcal{D}\colon C^\infty_+(M,S)\to C^\infty(M,S))
={}&
\mathrm{ind}\big(\mathcal{D},H^+\oplus V_>\big) =
\mathrm{ind}(\mathcal{D},V_>)+\dim\big(H^+\big) \\
={}&
\frac{1}{2} \big(\dim\big(H^+\big)-\dim(H^-)\big) \\
={}&
\frac{1}{2} \big(\dim\ker\big(\mathcal{D}^\partial\colon C^\infty\big(\partial M,S^+\big)\to C^\infty(\partial M,S^-)\big) \\
& -\dim\ker\big(\mathcal{D}^\partial\colon C^\infty(\partial M,S^-)\to C^\infty\big(\partial M,S^+\big)\big)\big) \\
={}&
\frac{1}{2} \mathrm{ind}\big(\mathcal{D}^\partial\colon C^\infty\big(\partial M,S^+\big)\to C^\infty(\partial M,S^-)\big) .
\end{align*}
This concludes the proof.
\end{proof}

The following consequence is known as cobordism invariance of the index:

\begin{Corollary}\label{cor:CobInv}
In addition to the assumptions in Theorem~{\rm\ref{thm:HolIndex}} assume that $S^\pm$ are the eigensubbundles of $S|_{\partial M}$ of the involution ${\rm i}\sigma_{\mathcal{D}}\big(\nu^\flat\big)$ for the eigenvalues $\pm 1$.
Then $S|_{\partial M} = S^+ \oplus S^-$ and the indices occurring in Theorem~{\rm\ref{thm:HolIndex}} vanish.
\end{Corollary}

\begin{proof}
Without loss of generality we can assume that $M$ is connected.
If $\partial M=\varnothing$, then the assertion is obvious.
Therefore, we assume $\partial M\neq \varnothing$.

We claim that the operators $\mathcal{D}\colon C^\infty_\pm(M,S)\to C^\infty(M,S)$ have trivial kernel.
To see this, suppose that $s\in C^\infty_\pm(M,S)\to C^\infty(M,S)$ satisfies $\mathcal{D}s=0$.
Since $\mathcal{D}$ is formally self-adjoint, we obtain
\begin{align*}
0
&=
\int_M \langle \mathcal{D}s,s\rangle - \int_M \langle s, \mathcal{D}s\rangle
=
\int_{\partial M} \big\langle \sigma_{\mathcal{D}}\big(\nu^\flat\big)s,s\big\rangle
=
\mp {\rm i}\int_{\partial M} |s|^2 .
\end{align*}
Hence, $s|_{\partial M}=0$.

We show that this implies $s=0$ on all of $M$.
By adding a small collar neighborhood to $M$ along $\partial M$ we embed $M$ into an open manifold $\tilde M$.
We extend the bundle $S$ and the Dirac-type operator $\mathcal{D}$ to $\tilde M$.
We extend $s$ by zero to $\tilde M$ and obtain a continuous section $\tilde s$.
Let $\phi$ be a~compactly support test section on $\tilde M$.
Then
\begin{align*}
\int_{\tilde M}\langle \tilde s,\mathcal{D}\phi\rangle
&=
\int_M\langle s,\mathcal{D}\phi\rangle
=
\int_M\langle \mathcal{D}s,\phi\rangle + \int_{\partial M} \big\langle s,\sigma_{\mathcal{D}}\big(\nu^\flat\big)\phi\big\rangle
=
0.
\end{align*}
This shows $\mathcal{D}\tilde s=0$ in the weak sense.
By elliptic regularity theory, $\tilde s$ is smooth and $\mathcal{D}\tilde s=0$ holds classically.
Now $\tilde s$ vanishes on a nonempty open subset of $\tilde M$, $\tilde M$ is connected and Dirac-type operators have the unique continuation property, see, e.g., \cite[Theorem~8.2]{Booss-Wojciechowski}.
Thus, $\tilde s=0$ on all of $\tilde M$.
\end{proof}

We generalize Freed's Theorem~B in \cite{Freed} to Dirac-type operators.

\begin{Corollary}\label{cor:Freed}
Let $\mathcal{D}$, $\mathcal{D}^\partial$, $S$, $M$, and $\nu$ be as in Theorem~{\rm\ref{thm:HolIndex}}.
Let $N_1,\hdots,N_k$ denote the connected components of $\partial M$.
Suppose that $\varepsilon_1,\hdots,\varepsilon_k$ is a collection of integers in $\{1,-1\}$. We define a bundle $S^+$ over $\partial M$ so that the fiber of $S^+$ at a point $x \in N_j$ is the eigenspace of~\smash{${\rm i}\sigma_{\mathcal{D}}\big(\nu^\flat\big)$} with eigenvalue $\varepsilon_j$. Similarly, we define a bundle $S^-$ over $\partial M$ so that the fiber of $S^-$ at a point~${x \in \partial M}$ is the eigenspace of \smash{${\rm i}\sigma_{\mathcal{D}}\big(\nu^\flat\big)$} with eigenvalue $-\varepsilon_j$. Then $S|_{\partial M} = S^+ \oplus S^-$. Moreover,
\begin{align*}
\mathrm{ind}(\mathcal{D}\colon C^\infty_+(M,S)\to C^\infty(M,S))
&=
\sum_{\varepsilon_j=1} \mathrm{ind}\big(\mathcal{D}^\partial\colon C^\infty\big(N_j,S^+\big)\to C^\infty(N_j,S^-)\big) .
\end{align*}
\end{Corollary}

\begin{proof}
By Corollary~\ref{cor:CobInv},
\begin{align*}
\begin{aligned}
0
={}&
\sum_{\varepsilon_j=1} \mathrm{ind}\big(\mathcal{D}^\partial\colon C^\infty\big(N_j,S^+\big)\to C^\infty(N_j,S^-)\big) \\
& + \sum_{\varepsilon_j=-1} \mathrm{ind}\big(\mathcal{D}^\partial\colon C^\infty(N_j,S^-)\to C^\infty\big(N_j,S^+\big)\big),
\end{aligned}
\end{align*}
hence
\begin{align*}
0
={}&
\sum_{\varepsilon_j=1} \mathrm{ind}\big(\mathcal{D}^\partial\colon C^\infty\big(N_j,S^+\big)\to C^\infty(N_j,S^-)\big) \\
& - \sum_{\varepsilon_j=-1} \mathrm{ind}\big(\mathcal{D}^\partial\colon C^\infty\big(N_j,S^+\big)\to C^\infty(N_j,S^-)\big) .
\end{align*}
Therefore,
\begin{align*}
\mathrm{ind}\big(\mathcal{D}^\partial\colon C^\infty\big(\partial M,S^+\big)\to C^\infty(\partial M,S^-)\big)
&=
\sum_j \mathrm{ind}\big(\mathcal{D}^\partial\colon C^\infty\big(N_j,S^+\big)\to C^\infty(N_j,S^-)\big) \\
&=
2 \sum_{\varepsilon_j=1} \mathrm{ind}\big(\mathcal{D}^\partial\colon C^\infty\big(N_j,S^+\big)\to C^\infty(N_j,S^-)\big).
\end{align*}
The result now follows from Theorem~\ref{thm:HolIndex}.
\end{proof}

\subsection*{Acknowledgements}

The first named author was supported by DFG-SPP 2026 ``Geometry at Infinity''.
The second named author was supported by the National Science Foundation under grant DMS-2103573 and by the Simons Foundation.
The third named author was supported by DFG-SPP 2026 ``Geometry at Infinity'' and by Columbia University.
We would like to thank the anonymous referees for their insightful comments, which helped to improve the exposition.

\pdfbookmark[1]{References}{ref}
\LastPageEnding

\end{document}